\documentclass[11pt]{amsart}
\usepackage{graphicx, xcolor}
\usepackage{amssymb}
\usepackage{ dsfont }
\numberwithin{equation}{section}

\def\N{\mathbb{N}}
\def\R{\mathbb{R}}\def\C{\mathbb{C}}

\newcommand{\rdbrack}{]\!]}
\newcommand{\ldbrack}{[\![}

\newcommand{\bpsi}{\boldsymbol{\psi}}
\newcommand{\bphi}{\boldsymbol{\phi}}

\renewcommand\d{\partial}

\def\l{\lambda}

\def\epsilon{\varepsilon}
\def\e{\varepsilon}

\newcommand\br{\begin{rem}}
\newcommand\er{\end{rem}}
\newcommand\bp{\begin{pmatrix}}
\newcommand\ep{\end{pmatrix}}
\newcommand\be{\begin{equation}}
\newcommand\ee{\end{equation}}
\newcommand\ba{\begin{equation}\begin{aligned}}
\newcommand\ea{\end{aligned}\end{equation}}

\newtheorem{theorem}{Theorem}[section]
\newtheorem{proposition}[theorem]{Proposition}

\newtheorem{example}[theorem]{Example}
\newtheorem{remark}[theorem]{Remark}

\title{Slow dynamics in reaction-diffusion systems} 

\begin{document}

\maketitle

\begin{center}
MARTA STRANI\footnote{Universit\'e Paris-Diderot, Institut de Math\'ematiques de Jussieu-Paris Rive Gauche, Paris (France). E-mail addresses: \texttt{marta.strani@imj-prg.fr}, \texttt{martastrani@gmail.com}.}
\end{center}

\vskip1cm

\begin{abstract}
We consider a system of reaction-diffusion equations in a bounded interval of the real line,
with emphasis on the {\it metastable dynamics}, 
whereby the time-dependent solution approaches the steady state in an asymptotically 
exponentially long time interval as the viscosity coefficient $\varepsilon>0$ goes to zero. To rigorous describe such behavior, we analyze the dynamics of layered solutions localized far from the stable configurations of the system, and we derive an ODE for the position of the internal interfaces.
\end{abstract}

\begin{quote}\footnotesize\baselineskip 14pt 
{\bf Key words:} 
Metastability, slow motion, internal layers, reaction-diffusion systems. 
 \vskip.15cm
\end{quote}

\begin{quote}\footnotesize\baselineskip 14pt 
{\bf AMS subject classifications:} 35B36, 35B40,  35K45, 35K57
 \vskip.15cm
\end{quote}

\pagestyle{myheadings}
\thispagestyle{plain}
\markboth{M.STRANI}{SLOW DYNAMICS IN REACTION-DIFFUSION SYSTEMS}

\section{Introduction}

In this paper we study the long time dynamics   of solutions for a class of parabolic systems of the form
\begin{equation}\label{systemintro}
\partial_t \boldsymbol{u} = \mathcal F^\varepsilon[\boldsymbol{u}], 
\end{equation}
 where $\mathcal F^\varepsilon$ is a nonlinear differential operator of reaction-diffusion type that depends singularly on the parameter $\varepsilon$, meaning that $\mathcal F^0[\boldsymbol{u}]$ is of lower order. The unknown $\boldsymbol{u}$ belongs to the space $[L^2(I)]^n$ for some bounded interval $I \subset \R$. System \eqref{systemintro} is complemented with initial datum $\boldsymbol{u}_0(x) \in [L^2(I)]^n$ and appropriate boundary conditions.
 \vskip0.2cm
In particular, we are interested in describing the metastable behavior manifested by a special class of solutions to \eqref{systemintro}. Roughly speaking, a metastable dynamics appears when the time dependent solutions of an evolutive PDE exhibit a first transient phase where they are close to some non-stationary state before converging, in an extremely long time, to their asymptotic limits, namely a solution to $\mathcal F^\varepsilon[\boldsymbol{u}]=0$.  In other words, a dynamics is said to be metastable if there exists a first time scale of order $\mathcal O(1)$ in time where a pattern of internal interfaces is formed; once this structure is formed, it persists for an exponentially long time interval, whose size usually depends on the parameter $\varepsilon$. Last step of the dynamics is an (exponentially slow) motion of the solutions towards their asymptotic configuration.

As a consequence, two different time scales emerge: a first transient phase of order $\mathcal O(1)$, where the solutions develop a layered structure, and a subsequent long time phase characterizes by the slow convergence of such solutions to some stable configuration.

\vskip0.2cm
Slow motion of internal interfaces has been investigated for a number of different partial differential equations. To name just some of these results, we recall here \cite{KreiKrei86}, a pioneering article in the study of slow dynamics for viscous scalar conservation laws; here the authors consider the viscous Burgers equation in a bounded interval of the real line, namely
$$\partial_t u= \varepsilon \partial^2_{x} u -\partial_x f(u), \quad x \in I,$$
complemented with Dirichlet boundary conditions, proving that the eigenvalues of the linearized operator around an equilibrium configurations are real, negative, and have the following distribution with respect to $\varepsilon$
\begin{equation*}
\lambda_1^\varepsilon=\mathcal O(e^{-1/\varepsilon})\quad \textrm{and}\quad \lambda_k^\varepsilon<-\frac{c_0}{\varepsilon}<0\qquad\forall\,k\geq 2.
\end{equation*}
The presence of a first eigenvalue small with respect to $\varepsilon$ implies that the convergence towards the equilibrium is very slow, when $\varepsilon$ is small: indeed,  the large time  behavior of solutions is described by terms of order $e^{\lambda_1^\varepsilon\,t}$.

Starting from this result, there are several papers concerning slow motion of internal shock layer for viscous conservation laws, see for instance  \cite{LafoOMal94}, \cite{LafoOMal95}, \cite{MS} and \cite{ReynWard95}. Slow dynamics of interfaces has been examined also for convection-reaction-diffusion equations (see  \cite{BerKamSiv01,Str13} and \cite{SunWard99}), for relaxation system, with the contribution \cite{Str12}, and for phase transition problem  described by the Cahn-Hilliard equation in \cite{AlikBatFus91, Pego89}. 

\vskip0.2cm
In this paper we mean to analyze the slow motion of internal interfaces for a system of reaction diffusion equations. Depending on the assumptions on the terms of the system, reaction-diffusion type systems can be used to describe different models in mathematics and physics: the principle areas of application, together with a few of the many possible references, are the following: neurophysics and biophysics \cite{Kee80, OstYan75}, chemical physics \cite{Fif76b, OrtLos75}, phase changes \cite{CagFif}, population genetics \cite{AroWei75} and mathematical ecology \cite{MimMur79}.
Most of the above applications deal with a system of equations, one of which is a reaction diffusion equation for the unknown $u$, with the source term depending on $u$ and another variable $v$, that satisfies a different equation coupled to $u$. 


\vskip0.2cm
If $n=1$, that is $u \in L^2(I)$, the one dimensional Allen-Cahn equation
 \begin{equation}\label{AC}
 \partial_t u = \varepsilon \partial^2_x u- f(u), \quad x \in I,
 \end{equation}
is the most popular example of reaction-diffusion equation. Equation \eqref{AC} is complemented with appropriate Dirichlet or Neumann boundary conditions and initial datum $u_0(x)$.  The function $f : \R \to \R$ is the derivative of a symmetric double well potential $W(u)$ with two equal non-degenerate minima. To fix the ideas, one may assume $W(u)=\frac{(u^2-1)^2}{4}$, which has two minima in $u=\pm 1$. 
 . 
 
 \vskip0.2cm
 Equation \eqref{AC} was firsty introduced by S.M. Allen and J.W. Cahn in \cite{AllCah79} to model the interface motion between different crystalline structures in alloys; in this context $u$ represents the concentration of one of the two components of the alloy and the parameter $\varepsilon$ is the interface width. A possible way to derive the Allen-Cahn equation is to compute the gradient of the Ginzburg-Landau energy functional
\begin{equation}\label{GL}
\mathcal I(u)= \int_0^1 \left(\frac{1}{2}\varepsilon |{\partial_x u}|^2 +W(u)\right) \, dx
\end{equation}
 into the space $L^2(I)$, where $W(u)$ is defined as $\partial_u W(u)=f(u)$. This point of view implies that we are assuming the field $u$ to be governed by a gradient equation on the form $\partial_t u=-\delta_u \mathcal I$ (see \cite{Fif02}), that is one assumes the solutions to the Allen-Cahn equation \eqref{AC} to evolve according to the $L^2$ gradient flow of the energy functional \eqref{GL}.
Precisely, there are two different effects: the reaction term $f(u)$ pushes the solution towards the two minima $u=\pm 1$, while the diffusion term  $\varepsilon \partial^2_x u$ tends to regularize and smoothen the solution. In the small viscosity limit, i.e. $\varepsilon \to 0$, two different phases appear, corresponding to regions where the solution is close to $\pm 1$, and the width of the transition layers  between these two phases is of order $\varepsilon$. In \cite{Evans92} it is proven that, in the limit $\varepsilon \to 0$, the solutions only attains the values $\pm 1$ and the interface evolves according to the motion by mean curvature.

  \vskip0.2cm
The steady state solutions to \eqref{AC} have been studied in several papers: for example, \cite{Fif74} for solutions with a single transition layer  in the whole line (and \cite{Fif76} for the same problem in a bounded interval with Dirichlet boundary conditions); the problem of the stability of the steady state has been considered in \cite{ClePel} in the case of finite domains with Neumann boundary conditions, but the most complete result concerning existence, uniqueness and stability for both simple and multi-layered solutions to \eqref{AC} in a one dimensional bounded domain has been proven in \cite{AngMalPel87}.
 \vskip0.2cm
For the time dependent problem, it is proven in \cite{FifHsi88} that, given an initial datum $u_0$ that changes sign  inside the interval, the solution of the evolutive equation develops into a layered function $u^\xi$, where $\xi \in \R^N$ represents the position of the interfaces, being $N \geq 1$ the number of the layers. Moreover, the study of the eigenvalue problem obtained by linearizing around $u^\xi$ is crucial in the study of the motion of the layers (\cite{Che92,Che94}): it is possible to show that there exist exactly $N$ exponentially small eigenvalues, with $N$ equals to the number of the layers, while the rest of the spectrum is bounded away from zero uniformly with respect to $\varepsilon$ (see, for example \cite{AngMalPel87}).
 \vskip0.2cm
 Following the line of a spectral analysis, the phenomenon of metastability for the Allen-Cahn equation has been studied, among others, in \cite{CarrPego89,CarrPego90} and \cite{Che04,FuscHale89}; in particular, in \cite{CarrPego89}, the authors prove that the dynamics of the solutions to \eqref{AC} is the following: starting from two-phase initial data, a pattern of interfacial layers develops in a relatively short time far from the stable equilibrium, defined as the minimizer of the energy functional $\mathcal I(u)$.
Once the solution has reached this state, it changes extremely slow, with a time scale proportional to $e^{1/\varepsilon}$.  To rigorously describe such metastable dynamics, the authors build-up a  family of functions $u^\xi(x)$ which approximates a metastable state with $N$ transition layers, being $\xi=(\xi_1, . . . . , \xi_N )$ the $N$ layer positions;  they subsequent linearize the original equation around an element of the family, obtaining a system of ordinary differential equations for the layers positions $\xi_i$. 

In \cite{Che04}, the work of Carr and Pego \cite{CarrPego89} is extended by considering generic initial data (not only the ones that are $\mathcal O(e^{-1/\varepsilon})$ close to the metastable pattern).

As already stressed, the key of the analysis is to one side the spectral analysis of the linearized operator arising from the linearization around $u^\xi(x)$: precisely, it is shown that there exist exactly $N$ small (with respect to $\varepsilon$) eigenvalues, while all the other eigenvalues are negative and bounded away from zero uniformly with respect to $\varepsilon$.
The existence of, at least, one first small eigenvalue implies that the convergence towards the equilibrium can be extremely slow, depending on the size of $\varepsilon$. On the other side, the limiting behavior of the solutions to the original problem can be well understood in terms of an invariant manifold theory (see also \cite{Fus90}).

Another important contribution in the study of the metastable dynamics for reaction-diffusion equations is the work of F. Otto and M.G. Reznikoff \cite{OttRez06}, where the authors describes the slow motion of PDEs with gradient flow structure, with a particular attention to the one-dimensional Allen-Cahn equation: the idea here is to translate informations on the energy into informations on the dynamics of the solutions; in particular, the authors state and prove a Theorem that gives sufficient condition to be imposed on the energy so that the associated gradient flow exhibits a metastable behavior.

In the line of the energy methods it is worth to mention also the previous fundamental contribution of L. Bronsard and R.V. Kohn \cite{BroKoh90}.

\vskip0.2cm
A striking result in the study of the metastability for the Allen-Chan equation in dimension greater or equal to one is the reference \cite{AliFus96}: concerning spectral properties, it is expected an infinity of critical eigenvalues associated with the interface (see \cite{AliFus94} for a precise result in this direction). The study of the eigenvalue problem provides, as usual, informations on the stability of the interfaces, as shown in \cite{DemSch90}.

The existence, the asymptotic stability of the steady states and the study of the time dependent problem for different type of reaction-diffusion systems in different space dimensions have been extensively studied. To name some of these results, see, for example, \cite{MimTabHos80} for a two component
reaction-diffusion system in one dimensional space (see also  \cite{NisFuj86} for the internal layer problem), or \cite{AlaBroGui96} for a couple of  scalar reaction diffusion equations  with $x \in \R^2$.
\vskip0.2cm
Results relative to metastability for systems of reaction-diffusion equations appear to be rare (we recall here, \cite{HubeSerr96, KreiKreiLore08}), the main difficulty stemming from the fact that a spectral analysis in the case of a system of equations needs much more care. A recent contribution is the reference \cite{BetOrlSme11}, where the authors study the slow dynamics of solutions of a system of reaction-diffusion equations by means of the study of the evolution of the localized energy. Other papers concerning the slow motion of solutions to systems with gradient structure are \cite{BetSme13}, \cite{Kaletall01} and \cite{Ris08}, whose analysis is based entirely on energy methods.

\vskip0.2cm
In this context, the aim of this paper is to  rigorous describe the slow motion of a pattern of internal layers for a system of reaction-diffusion equations; precisely, we are interested in the case of a single internal layer (being the case of a larger number of interfaces a straightforward adaptation), and we want to describe the dynamics after this is formed.

\vskip0.2cm
The main difference with respect to previous works concerning metastability for  reaction-diffusion systems is that here we develop a general theory than can be applicable to a wider class of equations, not only the ones of gradient type. Precisely, our analysis lies on the study of the linearized system obtained from the linearization of the original equations around a steady state; by using an adapted version of the projection method and by exploiting the spectral properties of the linearized operator, we are able to derive an explicit equation for the layer position describing its slow motion. 

This {\it dynamical} approach, based on a deep use of the structure of the equation, yields to a very exact result for the speed rate of convergence of the solution, in contrast to the energy approach that gives only an upper bound for the speed (see, for example, \cite{Gra95}); we thus achieve a rigorous result via a different method with respect to the previous papers that have consider metastablity for gradient flow systems, where the key of the analysis is the study of the energy associated to the system.

In particular, the Allen-Cahn equation and reaction-diffusion systems with a gradient structure fit into this general framework.

\vskip0.2cm
Before describing in details the strategy we mean to use, let us briefly explain the dynamical behavior of a generic solution to  \eqref{AC}.

Heuristically, in the case of a single equation of reaction-diffusion type, the solution with a single transition layer  evolves as follow: since the layer interacts with its reflections with respect to the boundary points of the interval $I=(a,b)$, if the layer is closer to $a$, then it is attracted by its reflection with respect to $a$ and it moves towards $a$. Once the layer have reached a neighborhood of $x=a$ it suddenly disappears, and, for large times, the solution converges to one of the stable patternless solutions $u= u_\pm$, where $u_\pm$ are the minima of the energy $\mathcal I(u)$.

As a consequence, starting from an initial datum with only one simple zero inside the interval $I$, in the first stage of the dynamics an internal layer is formed in an $\mathcal O(1)$ time scale; the subsequent motion of this layer is exponentially slow until finally, for $\varepsilon$ sufficiently small,  it quickly disappears and the solution converges to either $u_+$ or $u_-$.  

The fact that the dynamics yields to a reduction of the number of the layers up to their totally disappearance, agrees with the fact that, if $u$ is the solution of the equation, then the number $n(t)$ of the zeroes of $u$ is non increasing in time (see \cite{Ang88} as a reference). 

\vskip0.2cm
In particular, such slow motion can be described in terms of the evolution of the location of interface, named here $\xi(t)$; precisely, we expect $\xi(t)$ to slowly evolve towards one of the wall of the interval, up to the total  disappearance of the layers, corresponding to the convergence of the solution ${u}$ to one of the equilibrium configuration of the system. We will refer with $\bar \xi$ to this final {\it equilibrium} configuration for the variable $\xi$.

\vskip0.2cm
Going deeper in details, in order to describe this behavior for a system of reaction-diffusion equations,  the strategy we use is the following:

\vskip0.1cm
{\bf 1.} We build-up a one parameter family of functions $\{ \boldsymbol{U}^\varepsilon(x;\xi)\}$, whose elements are close to a steady state for \eqref{systemintro} in a sense that will be specified later, and where the parameter $\xi$ usually describe the location of the interface.
\vskip0.1cm
{\bf 2.} In order to study the dynamics of solutions up to the formation of the internal interface, we linearize the original system around 
the family $\{ \boldsymbol{U}^\varepsilon\}$, by writing the solution $\boldsymbol{u}$ as the sum of an element of the family and a small perturbation $\boldsymbol{v}$.
\vskip0.1cm
{\bf 3.} Under suitable hypotheses, we state and prove two results (see Theorem \ref{teo1} and Proposition \ref{cor:metaL2}) concerning the study of the coupled system for the perturbation $\boldsymbol{v}$ and the parameter $\xi$, and describing the metastable behavior of the system under consideration.

\vskip0.2cm
In \cite{MS}, the authors firstly utilize such a technique to describe the slow motion of the solutions to a general class of parabolic systems: the key of their analysis is the linearization around a one-parameter family of approximate steady states (parametrized by $\xi(t)$, describing the motion along the family), and the subsequent analysis of the system obtained for the couple $(\xi,\boldsymbol{v})$, where $\boldsymbol{v}$ is the perturbative term. Dealing with this system brings into the analysis of the specific form of the quadratic terms arising from the linearization that, for a certain class of parabolic system (as, for example, system of viscous conservation laws) involve a dependence on the space derivative of the solution. This is the reason why the authors consider only an approximation of the complete nonlinear equations for the couple $(\xi,\boldsymbol{v})$, obtained by disregarding the quadratic terms in $\boldsymbol{v}$, so that the dependence  on the space derivative is canceled out.

\vskip0.2cm
The aim of the present paper is to show that, in the specific case of parabolic systems of reaction-diffusion type, it is possible to perform a study of the complete system for the couple $(\xi,\boldsymbol{v})$.
More precisely, when linearizing a reaction-diffusion type system around an approximate steady state (step {\bf 2}) the nonlinear terms arising from the linearization do not involve space derivatives, so that it is possible to obtain an $L^2$ estimate for the perturbation $\boldsymbol{v}$, without disregarding the nonlinear terms.

In particular, by keeping in the system for the variables $(\xi,\boldsymbol{v})$ the nonlinear dependence on the variable $\xi$, and by estimating also the higher order terms arized from the linearization,  
our analysis gives a complete description of the dynamics up to the formation of the interface, and far from its equilibrium configuration.
As a consequence, the two phases of the dynamics (formation of the interface and subsequent slow motion of the layered solution) are rigorously separated and  described.

\vskip0.2cm
We close this Introduction describing in details the contribution of the paper. In Section 2  we build up a one-parameter family of approximate steady states  $\{ \boldsymbol{U}^\varepsilon(x;\xi) \}$ and we linearize the original system around an element of the family. By using an adapted version of the projection method, we obtain a coupled system for the perturbation $\boldsymbol{v}$, defined as the difference $\boldsymbol{v}:= \boldsymbol{u}- \boldsymbol{U}^\varepsilon$, and for the parameter $\xi$, describing the slow motion of the interface; in particular, the equation for $\xi$ is obtained in such a way the first component of the perturbation (the one that has a very slow decay) is canceled out. Finally, in Section 3, we analyze the complete system for the couple $(\xi, \boldsymbol{v})$, and we state and prove the following Theorem, that gives a precise estimate for the $L^2$ norm of the perturbation.
\begin{theorem}\label{teointro}
For every $\boldsymbol{v}_0 \in [L^2(I)]^n$ 
and for every $t \leq T^\varepsilon$, there holds 
\begin{equation*}
|\boldsymbol{v}-\boldsymbol{z}|_{{}_{L^2}}(t) \leq C \left( |\Omega^\varepsilon|_{{}_{L^\infty}} +C \,  |\boldsymbol{v}_0|^2_{{}_{L^2}}\right),
\end{equation*}
where $\boldsymbol{z}$ is such that $|\boldsymbol{z}|_{{}_{L^2}} \leq C \, e ^{-c t}$,
while $|\Omega^\varepsilon|_{{}_{L^\infty}}$ is converging to zero as $\varepsilon \to 0$. In addition, the final time $T^\e$ is of order $|\Omega^\e|_{{}_{L^\infty}}^{-1}$, hence diverging to $+\infty$ as $\varepsilon\to 0$. 
\end{theorem}
Theorem \ref{teointro} states that the solution $\boldsymbol{v}$ can be decomposed as the sum of two functions, $\boldsymbol{v}=\boldsymbol{z}+\boldsymbol{R}$, where $\boldsymbol{z}$ is given explicitly in terms of the wave speed measure $\xi(t)$ (for more details see section 3) and has a very fast decay in time, while  the error $\boldsymbol{R}$ is estimated. This estimate can be used to decouple the original system and leads us to the statement of the Proposition \ref{cor:metaL2}, providing the following estimate for interface position
\begin{equation*}
|\xi(t)-\bar \xi| \leq   |\xi_0|e^{-\beta^\varepsilon t},
\end{equation*} 
where $\bar \xi $ indicates the asymptotic value for $\xi$,  and $\beta^\varepsilon >0$ converges to zero as $\varepsilon \to 0$.
This formula explicitly shows how the motion of the interface towards its equilibrium configuration is much slower as $\varepsilon$ becomes smaller.

\vskip0.2cm
We point out that usually the eigenvalue analysis of the linearized problem around the steady solution does not show  how to understand the transition from the metastable state to the final stable state. In this paper, Theorem \ref{teo1} and Proposition \ref{cor:metaL2} give a good qualitative explanation of this transition, providing an explicit estimate for the size of the parameter $\xi$, together with an expression for the speed rate of convergence of the time dependent solution towards the equilibrium. Also, an explicit expression for the time of such convergence is given, showing that the metastable phase of the dynamics persists for an exponentially (with respect to the parameter $\e$) long time interval. 

\section{The metastable dynamics}

Given $\ell>0$, $I:=(-\ell,\ell)$ and $n\in\N$, 
we consider a system of reaction-diffusion equations
\begin{equation}\label{RDsystem}
\partial_t \boldsymbol{u} = \varepsilon \partial^2_x \boldsymbol{u} - \boldsymbol{f} (\boldsymbol{u}), \quad \boldsymbol{u} (t,0)= \boldsymbol{u}_0 (x)
\end{equation}
for the unknown $\boldsymbol{u}(x,\cdot) \,:\,[0,+\infty)\to [L^2(I)]^n$, and $\boldsymbol{f}: [L^2(I)]^n \to \R^n$, $\boldsymbol{f} = (f_1(\boldsymbol{u}), . . . . , f_n(\boldsymbol{u}))$. System \eqref{RDsystem} is also complemented with appropriate boundary conditions.

We require $\boldsymbol{f}$  to be a $C^1$ function such that there exist at least two constant  stable equilibria for the system.

In the specific case of systems with a  gradient structure, the above hypothesis is satisfied if one assumes that  there exists a $C^2$ potential function $W: \R^n \to \R$ such that $\boldsymbol{f}(\boldsymbol{u})=\nabla W(\boldsymbol{u})$ and such that $W$ has two distinct global minima $\pm \boldsymbol{u}^*$, i.e.
\vskip.15cm
{\bf i.}   $W(\pm \boldsymbol{u^*} ) = \nabla W(\pm \boldsymbol{u^*})=0$,
\vskip.15cm
{\bf ii.} $W(\boldsymbol{u}) >0$ in a neighborhood of $\pm \boldsymbol{u^*}$.
\vskip0.15cm
\noindent In this case it is known (see, among others, \cite{AliFus08, Ste87} ) that the only possible stable equilibrium solutions to \eqref{RDsystem} are minimizers of the energy functional
\begin{equation*}
\mathcal{I}(\boldsymbol{u}) = \int_I\left( \frac{1}{2}\varepsilon |\partial_x \boldsymbol{u}|^2+W(\boldsymbol{u})\right) \, dx.
\end{equation*}
As usual, we define
\begin{equation*}
	\langle \, \boldsymbol{u}, \boldsymbol{ v} \, \rangle:=\int_{I} \boldsymbol{u}(x)\cdot \boldsymbol{ v}(x)\,dx,
	\qquad\qquad \boldsymbol{u},\boldsymbol{ v}\in [L^2(I)]^n
\end{equation*}
the scalar product in $[L^2(I)]^n$, where $\cdot$ denotes the usual scalar product in $\R^n$.
\vskip0.2cm
We are interested in studying the behavior of $\boldsymbol{u}$ in the vanishing viscosity limit $\varepsilon \to 0$, in order to show that, when $\varepsilon \sim 0$, a metastable behavior occurs for the solutions to  \eqref{RDsystem}. For later use, let us define the nonlinear differential operator
\begin{equation*}
\mathcal F^\varepsilon[\boldsymbol{u}] := \varepsilon \partial_x^2 \boldsymbol{u} -\boldsymbol{ f}(\boldsymbol{u}).
\end{equation*}
The strategy we use here is analogous to the one performed in \cite{MS}; given a one-dimensional open interval $J$, we define the one-parameter family $\{\boldsymbol{ U}^{\varepsilon}(\cdot;\xi)\,:\,\xi\in J\}$, $\boldsymbol{U}^\varepsilon=(U_1^\varepsilon, . . . . , U_n^\varepsilon)$, whose elements can be considered as an approximation for the stationary
solution to \eqref{RDsystem} with a single internal layer, in the sense that ${\mathcal F}^\varepsilon[\boldsymbol U^{\varepsilon}(\cdot;\xi)] \to 0$ as $\varepsilon\to 0$. 

Precisely, we assume that the term $\mathcal F^\varepsilon[\boldsymbol U^{\varepsilon}]$ belongs 
to the dual space of the continuous functions space $[C(I)]^n$ and that there exists a family of smooth
positive  functions $\Omega^\varepsilon=\Omega^\varepsilon(\xi)$, uniformly convergent
to zero as $\varepsilon\to 0$,  such that there holds
\begin{equation}\label{defOmegaeps}
	|\langle \boldsymbol{\psi}(\cdot),{\mathcal F}^\varepsilon[\boldsymbol{U}^{\varepsilon}(\cdot,\xi)]\rangle|
		\leq \Omega^\varepsilon(\xi)\,|\boldsymbol{\psi}|_{{}_{\infty}} 
		\qquad \forall\,\boldsymbol{\psi}\in [C(I)]^n,
\end{equation}
for any $\xi\in J$. Hence, we ask for each element of the family to satisfy the stationary problem up to an error that is small in $\varepsilon$, and that is measured by $\Omega^\varepsilon$. 

Additionally, we require that there exists $\bar \xi \in J$ such that $\boldsymbol{U}^\varepsilon(x;\bar \xi)$ is a stable stationary solution, that is ${\mathcal F}^\varepsilon[\boldsymbol{ U}^{\varepsilon}(x;\bar \xi)]=0$. In particular, the parameter $\xi$ describes the dynamics of the solutions along the family $\{\boldsymbol{U}^\varepsilon\}$, and characterizes the slow motion of the interface; hence, the convergence of $\xi$ towards $\bar \xi$ describes the convergence of a layered solution towards its asymptotic configuration.

\begin{example}
{\rm The most well know example of scalar reaction-diffusion equation is given by the Allen-Cahn equation
\begin{equation}\label{ACesempio}
\d_t u=\e^2 \d_x^2 u- f(u).
\end{equation}
As in all the standard examples, we here choose $f$ to be the derivative of the double well potential $W(u)=\frac{1}{4} (u^2-1)^2$ (which has two distinct minima in $\pm u^*=\pm 1$) and we complement \eqref{ACesempio} with Dirichlet boundary condition $u(\pm1)=\pm1$; hence, we are imposing at the boundary the function $u$ to be equal to the pure phases $\pm u^*$. 

When considering the stationary equation $\e^2 \d_x^2 u+ u-u^3=0$ it is possible to prove (see, for example, \cite{BerBraButPre}) that there exists a unique stationary solution which is globally attractive, and it is given by a monotone profile connecting the pure phases $\pm u^*$.

In order to build up  a generic element of the family $\{ U^\e(x;\xi) \}$, one possible idea is to consider, for some $\xi \in I$, the function obtained by gluing together  the stationary solutions in the intervals $(-1,\xi)$ and $(\xi,1)$ satisfying, respectively,  the left and the right boundary conditions and the additional request $U^\e(\xi)=0$. Precisely
\begin{equation}\label{approxU}
	U^{\varepsilon}(x;\xi)=\left\{\begin{aligned}
		& k_1 \tanh \left((\xi-x)/\sqrt{2}\varepsilon\right) &\qquad &{\rm in}\quad (-1,\xi) \\
		& k_2 \tanh\left((\xi-x)/\sqrt{2}\varepsilon\right) &\qquad &{\rm in}\quad (\xi,1),
           \end{aligned}\right.
\end{equation}
and, by imposing $U^\e(\pm 1)=\pm 1$, it turns out
\begin{equation*}
k_1 = \frac{e^{(1+\xi)}+ e^{-(1+\xi)}}{e^{(1+\xi)}- e^{-(1+\xi)}} \quad {\rm and} \quad k_2 = \frac{e^{(1-\xi)}+ e^{-(1-\xi)}}{e^{(1-\xi)}- e^{-(1-\xi)}}.
\end{equation*}
The function defined in \eqref{approxU} is a continuous function but its first derivative has a jump, so that a straightforward computation shows that ${\mathcal F}^\varepsilon[U^{\varepsilon}(x;\xi)]=\ldbrack \partial_x U^{\varepsilon}\rdbrack_{{}_{x=\xi}}\delta_{{}_{x=\xi}}$, being
\begin{equation*}
\ldbrack \partial_x U^{\varepsilon}\rdbrack_{{}_{x=\xi}} \sim |k_1-k_2| \sim \frac{e^{-\sqrt{2} \xi /\e} + e^{-\sqrt{2}/\e}}{e^{\sqrt{2}/\e}-e^{\sqrt{2}\xi/\e}},
\end{equation*}
and this term is going to zero as $\e \to 0$ either if $\xi>0$ or if $\xi <0$ (in this last case the leading order term is of the form $e^{-\frac{\sqrt{2}}{\e}(\xi+1)}$ with $\xi >-1$). In particular, with such a construction, assumption \eqref{defOmegaeps} is satisfied with $\Omega^\e(\xi) = \ldbrack \partial_x U^{\varepsilon}\rdbrack_{{}_{x=\xi}}$.

\vskip0.2cm
Another possible construction  for the approximate family $\{ U^\e \}$ is to look  for a solution to \eqref{ACesempio} of traveling wave's type, i.e.
\begin{equation*}
u(x,t)= \Phi^\e(x-C^\e(\xi) t),
\end{equation*}
satisfying the additional assumptions $v^\e(\pm 1)=\pm 1$ and $v^\e(\xi)=0$, for some $\xi \in (-1,1)$; such solution moves with speed $C^\e(\xi)$ and satisfies the boundary condition only if $t=0$. Going further, since ${\mathcal F}^\varepsilon[v^{\varepsilon}]= \d_t v^\e$, we get
\begin{equation*}
{\mathcal F}^\varepsilon[v^{\varepsilon}(\cdot;\xi)]= - C^\e(\xi) \cdot \d_x v^\e (\cdot;\xi). 
\end{equation*}
so that the behavior of ${\mathcal F}^\varepsilon[v^{\varepsilon}]$ is encoded in the behavior of the speed $C^\e$, which is expected to be small with respect to $\e$.

}

\end{example}

\vskip0.5cm
Going back to our analysis, in order to describe the dynamics of solutions localized far away from the stable configuration, we linearize the original system \eqref{RDsystem} around an element of the family $\{\boldsymbol{U}^\varepsilon\}$ by looking for a solution to the initial value problem \eqref{RDsystem} in the form 
\begin{equation*}
	\boldsymbol{u}(x,t)=\boldsymbol U^{\varepsilon}(x;\xi(t))+\boldsymbol{ v}(x,t),
\end{equation*} 
with $\xi=\xi(t)\in J$  and the perturbation $\boldsymbol{ v}=\boldsymbol{v}(x,t)\in [L^2(I)]^n$ to be determined. Essentially,  the problem is recast in a sort of slow-fast dynamics (being $\xi$ the slow variable and $v$ the fast one); since we mean to describe the dynamics up to the formation of the interface and we want to follow its evolution towards the asymptotic limit, we thus suppose that the parameter $\xi$ depends on time, so that its evolution  describes the asymptotic convergence of the interface towards the equilibrium. Substituting into \eqref{RDsystem}, we obtain
\begin{equation}\label{eqv}
	\partial_t \boldsymbol{ v}=\mathcal{L}^\varepsilon_\xi \boldsymbol{v}+{\mathcal F}^\varepsilon[\boldsymbol U^{\varepsilon}(\cdot;\xi)]
		-\partial_{\xi} \boldsymbol U^{\varepsilon}(\cdot;\xi)\,\frac{d\xi}{dt}+\boldsymbol{ \mathcal{Q}}^\varepsilon[\boldsymbol{v},\xi],
\end{equation}
where
\begin{equation*}
	\begin{aligned}
		\mathcal{L}^\varepsilon_\xi \boldsymbol{ v}&:=d{\mathcal F}^\varepsilon[\boldsymbol U^{\varepsilon}(\cdot;\xi)]\,\boldsymbol{ v},\\
		\boldsymbol{\mathcal{Q}}^\varepsilon[\boldsymbol{ v},\xi]&:={\mathcal F}^\varepsilon[\boldsymbol U^{\varepsilon}(\cdot;\xi)+\boldsymbol{ v}]
			-{\mathcal F}^\varepsilon[\boldsymbol U^{\varepsilon}(\cdot;\xi)]
				-d{\mathcal F}^\varepsilon[\boldsymbol U^{\varepsilon}(\cdot;\xi)]\, \boldsymbol{v}.
	\end{aligned} 
\end{equation*}
Precisely, in the specific case considered here, $\mathcal L^\varepsilon_\xi$ has the following expression
\begin{equation*}
\mathcal L^\varepsilon_\xi  \boldsymbol{v} := \varepsilon \partial_x^2  \boldsymbol{v} - \boldsymbol{f}'( \boldsymbol U^\varepsilon)  \boldsymbol{v}.
\end{equation*}
Hence, when $n=1$, it is easy to check that $\mathcal L^\varepsilon$ is self-adjoint and therefore its spectrum is composed by real eigenvalues. In the general case $n >1$, the self adjointness of the operator is not guaranteed, and we have to consider the chance of having complex eigenvalues.
\vskip0.2cm
Next, let us require the linear operator $\mathcal{L}^\varepsilon_\xi$ to have a discrete spectrum
composed by semi-simple eigenvalues $\{ \lambda_k^\varepsilon(\xi) \}_{k \in \N}$ with corresponding right eigenfunctions $\bphi^\varepsilon_k=\bphi^\varepsilon_k(\cdot;\xi)$. 
Moreover, we assume that for any $\xi\in J$ the first eigenvalue $\lambda_1^\varepsilon$ is simple, real and such that $\lim\limits_{\varepsilon \to 0} \lambda_1^\varepsilon (\xi) =0,$ uniformly with respect to $\xi$; additionally, we suppose there hold
\begin{equation}\label{assumpintroeigen}
\begin{aligned}
            \lambda^\varepsilon_1(\xi)-Re \,\lambda^\varepsilon_2(\xi) \geq C, \qquad
	 Re\,\lambda^\varepsilon_k(\xi)\leq -C\quad \textrm{for }k\geq 2,
	\end{aligned}
\end{equation}
for some constant $C>0$ independent on $k$, $\varepsilon$ and $\xi$.

\begin{remark}
\rm{We note that there are no requests on the sign of the first eigenvalue $\l^\e_1$. Indeed, the metastable behavior is a consequence only of the smallness, with respect to $\e$, of the absolute value of such first eigenvalue. Also, assumption \eqref{assumpintroeigen}$_1$ states that there is only one critical eigenvalue, and it is compatible with the case of a single transition layer we are considering in this paper. In the general case of $N \geq 1$ layers, it has been shown in \cite{CarrPego89} that there exist exactly $N$ small (with respect to $\varepsilon$) eigenvalues. 

}
\end{remark}

If we now denote with  $\bpsi^\varepsilon_k=\bpsi^\varepsilon_k(\cdot;\xi)$ the eigenfunctions of  the adjoint operator 
$\mathcal{L}^{\varepsilon,\ast}_\xi$, we can define
\begin{equation*}
	v_k=v_k(\xi;t):=\langle \bpsi^\varepsilon_k(\cdot;\xi),\boldsymbol{v}(\cdot,t)\rangle.
\end{equation*}
Since we have assumed the first eigenvalue of the linearized operator to be small in $\e$, i.e. $\lambda_1^\e \to 0$ as $\e \to 0$, a necessary condition for the solvability of \eqref{eqv} is that  the first component of the solution $v_1$ has to be zero. This request translates into an equation for the parameter $\xi(t)$, chosen in such a way that the unique growing terms in the perturbation $\boldsymbol{ v}$ are canceled out. Precisely, we require
\begin{equation*}
	\frac{d}{dt} \langle \bpsi^\varepsilon_1(\cdot;\xi(t)), \boldsymbol{ v}(\cdot,t) \rangle =0
	\qquad\textrm{and}\qquad
	\langle \bpsi^\varepsilon_1(\cdot;\xi_0), \boldsymbol{ v}_0(\cdot)\rangle=0.
\end{equation*}
Using equation \eqref{eqv}, and since $\langle \bpsi^\varepsilon_1, {\mathcal L}^\varepsilon_{\xi} \boldsymbol{v} \rangle= \lambda^\varepsilon_1\langle \bpsi^\varepsilon_1, \boldsymbol{ v} \rangle =0$, 
we obtain a scalar differential equation for the variable $\xi$, describing the dynamics in a neighborhood of the approximate family, that is
\begin{equation}\label{eqxi0}
		\alpha^\varepsilon(\xi,\boldsymbol{v})\frac{d\xi}{dt}=\langle \bpsi^\varepsilon_1(\cdot;\xi),
			{\mathcal F}[\boldsymbol{U}^{\varepsilon}(\cdot;\xi)]+\boldsymbol{\mathcal{Q}}^\varepsilon[\boldsymbol{ v},\xi] \rangle,
\end{equation}
where
\begin{equation*}	
	\alpha^\varepsilon(\xi,\boldsymbol{v})
			:= \langle \bpsi^\varepsilon_1(\cdot;\xi), \partial_{\xi}\boldsymbol{U}^{\varepsilon}(\cdot;\xi) \rangle - \langle  \partial_{\xi}\bpsi^\varepsilon_1(\cdot;\xi),\boldsymbol{ v} \rangle,
\end{equation*}
together with the condition on the initial datum $\xi_0$
\begin{equation*}
	\langle \bpsi^\varepsilon_1(\cdot;\xi_0), \boldsymbol{v}_0(\cdot) \rangle =0.
\end{equation*}
For the sake of simplicity, we renormalize the first eigenfunction 
$\bpsi^\varepsilon_1$ in such a way
\begin{equation*}
	\langle \bpsi^\varepsilon_1(\cdot;\xi), \partial_{\xi}\boldsymbol{U}^{\varepsilon}(\cdot;\xi)\rangle=1.
\end{equation*}
Such constraint can be imposed if we ask for
\begin{equation*}
|\langle \bpsi^\varepsilon_1(\cdot;\xi), \partial_{\xi}\boldsymbol{U}^{\varepsilon}(\cdot;\xi)\rangle| \geq c_0 >0,
\end{equation*}
where $c_0$ is independent on $\xi$. The last assumption gives a (weak) restriction on the choice of the  family of functions $\{ \boldsymbol{U}^\varepsilon \}$; indeed, we need the manifold to be never transverse to the first eigenfunction $\bpsi^\varepsilon_1$.

Going further, in the regime of small $\boldsymbol{v}$, we have
\begin{equation*}
	\frac{1}{\alpha^\varepsilon(\xi,\boldsymbol{v})} = 1+
		\langle \partial_{\xi} \bpsi^\varepsilon_1,\boldsymbol{v} \rangle + R^\varepsilon[\boldsymbol{ v}],
		\end{equation*}
where the reminder $R^\varepsilon$ is of order $o(|\boldsymbol{ v}|)$, and it is defined as
\begin{equation*}
R^\varepsilon[\boldsymbol{v}]:= \frac{\langle  \partial_{\xi}\bpsi^\varepsilon_1(\cdot;\xi),\boldsymbol{v} \rangle^2}{1-\langle  \partial_{\xi}\bpsi^\varepsilon_1(\cdot;\xi),\boldsymbol{ v} \rangle}.
\end{equation*}
Hence, we may rewrite the nonlinear equation for $\xi$ as
\begin{equation}\label{eqxiNL}
	\frac{d\xi}{dt}=\theta^\varepsilon(\xi)\bigl(1+\langle\partial_{\xi} \bpsi^\varepsilon_1, \boldsymbol{v} \rangle\bigr)
		+ \rho^\varepsilon[\xi,\boldsymbol{v}], 
\end{equation}
where
\begin{equation*}
	\begin{aligned}
 	\theta^\varepsilon(\xi)
		&:=\langle \bpsi^\varepsilon_1,{\mathcal F[\boldsymbol{U}^{\varepsilon}] \rangle},\\
	\rho^\varepsilon[\xi,\boldsymbol{ v}]&:=\frac{1}{\alpha^\varepsilon(\xi,\boldsymbol{ v})}
		\bigl(\langle \bpsi^\varepsilon_1,\boldsymbol{\mathcal{Q}}^\varepsilon\rangle
		+\langle \partial_\xi\bpsi^\varepsilon_1, \boldsymbol{v}\rangle^2\bigr).
	\end{aligned}
\end{equation*}
Equation \eqref{eqxiNL} is an equation of motion for the parameter, describing the dynamics of $\xi(t)$  around the family of approximate steady states; since an element of the family $\{ \boldsymbol{U}^{\varepsilon}(x;\xi)\}_{\xi \in J}$ is not an exact steady state to \eqref{RDsystem}, the dynamics walks away from $\boldsymbol{U}^{\varepsilon}$ with a speed dictated by \eqref{eqxiNL}. In particular, by taking advantage of the existence of a small eigenspace, we  reduce the PDE to an ODE for the  parameter $\xi$.

We observe that, at a first approximation, for small perturbations $\boldsymbol{v} \sim 0$, the leading order term in the equation for $\xi(t)$, given by $\theta^\varepsilon(\xi)$, characterizes the speed of the solution during its motion towards its equilibrium configuration. From the definition of $\theta^\varepsilon$ we can formally see that such convergence is much slower as $\varepsilon$ becomes smaller, since $\theta^\varepsilon$ is going to zero as $\varepsilon \to 0$.

Using \eqref{eqxiNL}, we can rewrite the equation for $\boldsymbol{v}$ as
\begin{equation}\label{eqvNL}
	\partial_t \boldsymbol{v}=\boldsymbol{ H}^\varepsilon(x;\xi)
		+ ({\mathcal L}^\varepsilon_\xi+{\mathcal M}^\varepsilon_\xi)\boldsymbol{ v}
			+\boldsymbol{\mathcal{R}}^\varepsilon[\boldsymbol{v},\xi],
\end{equation}
where 
\begin{align*}
		\boldsymbol{H}^\varepsilon(\cdot;\xi)&:={\mathcal F}^\varepsilon[\boldsymbol{U}^{\varepsilon}(\cdot;\xi)]
			-\partial_{\xi}\boldsymbol{U}^{\varepsilon}(\cdot;\xi)\,\theta^\varepsilon(\xi),\\
		\mathcal M^\varepsilon_\xi \boldsymbol{v}&:=-\partial_{\xi}\boldsymbol{U}^{\varepsilon}(\cdot;\xi)
			\,\theta^\varepsilon(\xi)\,\langle\partial_{\xi} \bpsi^\varepsilon_1, \boldsymbol{ v} \rangle,\\ 
		\boldsymbol{\mathcal R}^\varepsilon[\boldsymbol{v},\xi]&:=\boldsymbol{\mathcal{Q}}^\varepsilon[\boldsymbol{ v},\xi] -\partial_{\xi}\boldsymbol{U}^{\varepsilon}(\cdot;\xi)\,\rho^\varepsilon[\xi,\boldsymbol{ v}],
\end{align*}
obtaining the couple system \eqref{eqxiNL}-\eqref{eqvNL} for the parameter $\xi$ and the perturbation $v$.

\begin{remark}{\rm
We observe that, in the case of $N >1$ internal layers the parameter $\xi \in \R^N$ and,  as proven in \cite{CarrPego89},  there exist exactly $N$ small (with respect to $\varepsilon$) eigenvalues. In particular, it is possible to adapt the procedure described above by projecting the equation into the  correspondingly $N$ small eigenspaces, obtaining exactly $N$ equations of motion for $\xi_i(t)$, $i=1, \dots , N$.

}
\end{remark}

\begin{example}\label{ExAC}{\rm

{\bf The case n=1.} Let us consider again the initial boundary value problem for Allen-Cahn equation in a one-dimensional interval $I=(-\ell,\ell)$, that is
\begin{equation}\label{cauchyAC}
 \left\{\begin{aligned}
		&\partial_t u 	 =\varepsilon\,\partial_x^2u -f(u)
		&\qquad &x\in I, t \geq 0\\
		&u(x,0)		 =u_0(x) &\qquad &x\in I,
 	 \end{aligned}\right.
\end{equation}
complemented with either Dirichlet or Neumann boundary condition.  As usual, we assume $f(u)=W'(u)$, where $W(u)$ is a double well function with equal minima in $u=\pm 1$. We require
\begin{equation*}
W(\pm 1)=W'(\pm 1)=0, \quad W(u) >0 \ \ {\rm for} \ \ -1< u <1
\end{equation*}
A standard example is given by  $W(u)=\frac{(u^2-1)^2}{4}$, which has two minima in $u=\pm 1$.

Given a reference state $U^\varepsilon(x;\xi)$ such that the nonlinear term
$
\mathcal F^\varepsilon[U^\varepsilon]=\varepsilon \partial_x^2 U^\varepsilon-f(U^\varepsilon)$ is going to zero as $\varepsilon \to 0$ in the sense of \eqref{defOmegaeps}, if we linearize around $U^\varepsilon$ we obtain the following equation for the perturbation 
\begin{equation*}
\partial_t v = \varepsilon \partial^2_x v - f'(U^\varepsilon) v - \partial_\xi U^\varepsilon \frac{d\xi}{dt} + \mathcal F^\varepsilon[U^\varepsilon] + \mathcal Q^\varepsilon[\xi,v],
\end{equation*}
where $ \mathcal Q^\varepsilon[\xi,v]:=-f''(U^\varepsilon) v^2-f'''(U^\varepsilon)v^3$. In particular, since we are assuming the perturbation $v$ to be small, i.e. $v \ll 1$, it follows that $|\mathcal Q^\varepsilon|_{{}_{L^1}} \leq C|v|^2_{{}_{L^2}}$.

We note that, in this case, the parameter $\xi$ can be chosen as the position of the (unique) interface, and its motion characterizes the slow dynamics of the  layered solution towards the patternless equilibrium.

\vskip0.2cm

{\bf The case n=2.} Let us consider a system of two reaction-diffusion equations, namely
\begin{equation}\label{ex1}
\left\{\begin{aligned}
\partial_t u &= \varepsilon \partial^2_x u - f_1(u,v) \\
\partial_t v &= \varepsilon \partial^2_x v - f_2(u,v)
\end{aligned}\right.
\end{equation}
and let $\boldsymbol U^\varepsilon=(U^\varepsilon, V^\varepsilon)$ be an approximate steady state of the problem, meaning that
\begin{equation*}
\mathcal F^\varepsilon [\boldsymbol U^\varepsilon] := \left( \begin{aligned}  \varepsilon \partial^2_x U^\varepsilon& - f_1(U^\varepsilon,V^\varepsilon) \\
\varepsilon \partial^2_x U^\varepsilon& - f_2(U^\varepsilon,V^\varepsilon) \end{aligned}\right) \end{equation*}
is small in $\varepsilon$ in the sense of \eqref{defOmegaeps}. By susbstituing 
$$(u,v) (x,t)= (w,z)(x,t) + (U^\varepsilon(x;\xi(t)), V^\varepsilon(x;\xi(t)))$$
into \eqref{ex1}, and by expanding $f_{1,2}(w+ U^\varepsilon, z+V^\varepsilon)$, we obtain the equation \eqref{eqv} for the perturbation $\boldsymbol{ v}=(w,z)$, where
\begin{equation*}
\mathcal L^\varepsilon_\xi \boldsymbol{v}:= \left( \begin{aligned} \varepsilon \partial_x^2 w &-\partial_u f_1(U^\varepsilon, V^\varepsilon) w - \partial_v f_1 (U^\varepsilon, V^\varepsilon) z \\
\varepsilon \partial_x^2 z &-\partial_u f_2(U^\varepsilon, V^\varepsilon) w - \partial_v f_2 (U^\varepsilon, V^\varepsilon) z
\end{aligned}\right),
\end{equation*}
and the nonlinear term is explicitly given by
\begin{equation*}
\boldsymbol{\mathcal Q}^\varepsilon [\xi, \boldsymbol{ v}] := \left( \begin{aligned}
-\frac{1}{2}\partial_u^2 f_1(U^\varepsilon,V^\varepsilon) w^2 &- \frac{1}{2}\partial_v^2 f_1(U^\varepsilon,V^\varepsilon) z^2 - \partial_{uv}^2 f_1 (U^\varepsilon;V^\varepsilon) w z \\
-\frac{1}{2}\partial_u^2 f_2(U^\varepsilon,V^\varepsilon) w^2 &- \frac{1}{2}\partial_v^2 f_2(U^\varepsilon,V^\varepsilon) z^2 - \partial_{uv}^2 f_2 (U^\varepsilon;V^\varepsilon) w z
\end{aligned}\right),
\end{equation*}
where the third order terms are omitted since $v \ll 1$.
Again, such formula explicitly shows that $|\boldsymbol{\mathcal Q}^\varepsilon|_{{}_{L^1}} \leq C  |(u,v)|^2_{{}_{L^2}}$.

 From this example one can easily check that, when $n >1$, the operator $\mathcal L^\varepsilon_\xi$ is not necessarily self-adjoint.

\vskip0.5cm
For a general $n \in \N$, the nonlinear term $\boldsymbol{\mathcal Q}^\varepsilon [\xi, \boldsymbol{ v}] $ is defined as
\begin{equation*}
\boldsymbol{\mathcal Q}^\varepsilon [\xi, \boldsymbol{ v}] :=\left( \begin{aligned}
- \sum_{i=1}^n \sum_{j=1}^n \partial_{u_i \, u_j}^2 & f_1 (\boldsymbol{U}^\varepsilon) v_i v_j \\
 .& \\
  .& \\
   .& \\
    .& \\
-  \sum_{i=1}^n \sum_{j=1}^n \partial_{u_i \, u_j}^2 & f_n (\boldsymbol{U}^\varepsilon) v_i v_j  
 \end{aligned}\right)
\end{equation*}
\vskip0.2cm
\noindent so that $|\boldsymbol{\mathcal Q}^\varepsilon|_{{}_{L^1}} \leq C |\boldsymbol{v}|^2_{{}_{L^2}}$ for some constant $C$ depending on $\boldsymbol f $ and $|\boldsymbol{U}^\varepsilon|$.

}
\end{example}

\section{Nonlinear metastability}

Let us consider the system  \eqref{eqxiNL}--\eqref{eqvNL} for the couple $(\xi,\boldsymbol{v})$
\begin{equation}\label{NLS}
 	\left\{\begin{aligned}
	\frac{d\xi}{dt}&=\theta^\varepsilon(\xi)\bigl(1 
		+\langle\partial_{\xi} \psi^\varepsilon_1, \boldsymbol{v} \rangle\bigr)+ \rho^\varepsilon[\xi,\boldsymbol{v}], \\
	\partial_t \boldsymbol{v} &=\boldsymbol{ H}^\varepsilon(\xi)+ ({\mathcal L}^\varepsilon_\xi+{\mathcal M}^\varepsilon_\xi)\boldsymbol{v}+ \boldsymbol{\mathcal R}^\varepsilon[\xi,\boldsymbol{v}]
 	\end{aligned}\right. 
\end{equation}
complemented with initial conditions
\begin{equation}\label{initialLS}
	\xi(0)=\xi_0\in J \qquad\textrm{and}\qquad \boldsymbol{v}(x,0)=\boldsymbol{ v}_0(x):=\boldsymbol{u}_0-\boldsymbol{U}^\varepsilon(\cdot;\xi_0)\in [L^2(I)]^n.
\end{equation}
In what follow, we mean to analyze the solutions to \eqref{NLS}. The main difference with respect to \cite{MS}, is that here we consider the complete system for the couple $(\xi,\boldsymbol{v})$, where also  the nonlinear terms $\boldsymbol{\mathcal R}^\varepsilon$ and $\mathcal \rho^\varepsilon$ are taken into account: as already stressed, in the particular case of reaction-diffusion system, these terms depend only from the perturbation $\boldsymbol{v}$, and not from its derivatives, so that an estimate for the $L^2$ norm of the perturbation  can be performed (sometimes  under appropriate smallness assumptions on the initial datum) without disregarding any higher order term.

\subsection{Estimates for the perturbation $v$}

Before stating our result, we recall the hypotheses we require on the terms of the system. 

\vskip.15cm
{\bf H1.} There exists a family of smooth positive functions $\Omega^\varepsilon$ such that
\begin{equation*}
	|\langle \bpsi(\cdot),{\mathcal F}^\varepsilon[\boldsymbol{U}^{\varepsilon}(\cdot,\xi)]\rangle|
		\leq \Omega^\varepsilon(\xi)\,|\bpsi|_{{}_{L^\infty}} \qquad \forall\,\bpsi\in [C(I)]^n,
\end{equation*}
with $\Omega^\varepsilon$ converging to zero as $\varepsilon\to 0$, uniformly with respect to $\xi$.
\vskip.15cm

{\bf H2.} Let  $\{ \lambda_k^\varepsilon(\xi) \}_{k \in \N}$ 
be the sequence of semi-simple eigenvalues of the linear operator  $\mathcal{L}^\varepsilon_\xi$.
Assume that for any $\xi\in J$, the first eigenvalue $\lambda_1^\varepsilon$ is simple, real and  such that $\lambda_1^\varepsilon(\xi) \to 0$ as $\varepsilon\to 0$ uniformly with respect to $\xi$. Moreover,  there hold
\begin{equation*}
         \lambda^\varepsilon_1(\xi)-Re \,\lambda^\varepsilon_2(\xi) \geq C, \qquad
	Re \, \lambda^\varepsilon_k(\xi)\leq -C\quad \textrm{for }k\geq 2,
\end{equation*}
for some constant $C>0$ independent on $k$, $\varepsilon$ and $\xi$.
\vskip.15cm

{\bf H3.} Given $\xi\in J$, let $\bphi^\varepsilon_k(\cdot;\xi)$ and $\bpsi^\varepsilon_k(\cdot;\xi)$
be a sequence of eigenfunctions for the operators $\mathcal{L}^\varepsilon_{\xi}$ and 
$\mathcal{L}^{\varepsilon,\ast}_{\xi}$ respectively, normalized so that $\langle \bpsi^\varepsilon_j, \bphi^\varepsilon_k \rangle = \delta_{jk}$; we assume 
\begin{equation}\label{derpsiphi}
	\sum_{j} \langle \partial_\xi \bpsi^\varepsilon_k, \bphi^\varepsilon_j\rangle^2
	=\sum_{j} \langle \bpsi^\varepsilon_k, \partial_\xi \bphi^\varepsilon_j\rangle^2
	\leq C,
\end{equation}
for all $k$ and for some constant $C$ independent on the parameter $\xi$. 
\vskip.15cm

\begin{remark}{\rm
Hypothesis {\bf H2} is a crucial hypothesis concerning the distribution of the eigenvalues of the linearized operator around an (approximate) steady state; again, we underline that there are no request on the sign of the first eigenvalue, while it is extremely important its smallness with respect to $\e$ and the presence of a spectral gap, encoded in the request $ \lambda^\varepsilon_1(\xi)-Re \,\lambda^\varepsilon_2(\xi) \geq C$. In other words, what it is crucial is the fact that there exists a finite number of small eigenvalues (which is equals to the number of the layers), while all the other eigenvalues are negative and bounded away from zero. This property indeed will translate into the fact that all the components of the perturbation except the first one will have a very fast decay in time, and in a slow motion for the single internal interface as a consequence of the size of $\lambda_1^\e$.

 For example, in the case of the Allen-Cahn equation \eqref{cauchyAC} with a single transition layer and with Neumann boundary conditions, it is proven (\cite{AliFus96, Bar10} and \cite{CarrPego89} for a linearization around an approximate steady state) that the spectrum is composed by real eigenvalues with the following distribution
\begin{equation*}
 \lambda_1^\varepsilon = \mathcal O( e^{-c_1/\varepsilon})  \quad {\rm and} \quad \lambda_k^\varepsilon \leq -c_2 \quad \forall k \geq 2,
\end{equation*}
where the constants $c_2,c_2>0$ are independent on $\varepsilon$. In particular, since the metastable layered solution is an unstable configuration for the system,  $\lambda_1^\varepsilon >0$. 

For systems, the situation is more delicate (see, as an example, \cite{NisFuj86}), and a spectral analysis is needed in order to verify  hypothesis {\bf H2}. In this direction, we quote here the analysis of \cite{Str12}.  Conversely, when rigorous results are not to be realizable, it could be possible to obtain numerical evidence of the spectrum of the linearized operator.

Let us also underline the importance of  the assumption of a {\it discrete spectrum}  composed by  {\it semi-simple eigenvalues}: indeed,  it will be crucial in the proof of the following Theorem (for more details, see the subsequent Remark 3.4).
}
\end{remark}

\begin{theorem}\label{teo1}
Let  the couple $(\xi,v)$ be the solution to the initial-value problem \eqref{NLS}-\eqref{initialLS}. If the hypotheses {\bf H1-2-3} are satisfied, then, for every $\varepsilon$ sufficiently small 
and for every $t \leq T^\varepsilon$, there holds for the solution $\boldsymbol{v}$
\begin{equation*}
|\boldsymbol{v}-\boldsymbol{z}|_{{}_{L^2}}(t) \leq C \left( |\Omega^\varepsilon|_{{}_{L^\infty}} +\exp\left(\int_0^t \lambda_1^\varepsilon(\tau) \, d\tau\right) |\boldsymbol{v}_0|^2_{{}_{L^2}}\right),
\end{equation*}
where the function $\boldsymbol{z}$ is defined as
\begin{equation*}
\boldsymbol{z}(x,t):=\sum_{k\geq 2} v_k(0)\,\exp \left( \lambda_k^\varepsilon(\xi(\tau)) \right)\,\bphi^\varepsilon_k(x;\xi(t)).
\end{equation*}
Moreover, the time $T^\varepsilon$ is of order $\ln |\Omega^\varepsilon|^{-1}_{{}_{L^\infty}}\,  \sup\limits_{\xi\in J} |\lambda_1^\varepsilon (\xi)|^{-1}$, hence diverging to $+\infty$ as $\varepsilon \to 0$.

\end{theorem}

\begin{proof}

Setting
\begin{equation}\label{normalform}
	\boldsymbol{v}(x,t)=\sum_{j} {v}_j(t)\,\bphi^\varepsilon_j(x,\xi(t)),
\end{equation}
we obtain an infinite-dimensional differential system for the coefficients $v_j$
\begin{equation}\label{eqwk_bis}
	\frac{dv_k}{dt}=\lambda^\varepsilon_k(\xi)\,v_k
		+\langle \bpsi^\varepsilon_k,\boldsymbol{F}\rangle + \langle \bpsi^\varepsilon_k,\boldsymbol{G}\rangle
\end{equation}
where, omitting the dependencies for shortness,
\begin{equation*}
\begin{aligned}
	\boldsymbol{F}&:=\boldsymbol{H}^\varepsilon+\sum_{j} v_j\,\Bigl\{{\mathcal M}^\varepsilon_\xi\, \bphi^\varepsilon_j
			-\partial_\xi \bphi^\varepsilon_j\,\frac{d\xi}{dt}\Bigr\}
		=\boldsymbol{H}^\varepsilon-\theta^\varepsilon \sum_{j}\Bigl(\boldsymbol{a}_j+\sum_{\ell} \boldsymbol{b}_{j\ell}\,v_\ell\Bigr)v_j, \\
	\boldsymbol{G}&:= \boldsymbol{\mathcal Q}^\varepsilon- \left( \sum_j \partial_\xi \bphi_j^\varepsilon v_j + \partial_\xi \boldsymbol{U}^\varepsilon\right) \left \{  \frac{\langle \bpsi_1^\varepsilon,\boldsymbol{ \mathcal Q}^\varepsilon \rangle}{1-\langle \partial_\xi \bpsi_1^\varepsilon, \boldsymbol{v} \rangle }+ \theta^\varepsilon \frac{\langle \partial_\xi \bpsi_1^\varepsilon, \boldsymbol{v} \rangle^2}{1-\langle \partial_\xi \bpsi_1^\varepsilon, \boldsymbol{v} \rangle }  \right\}= \boldsymbol{\mathcal Q}^\varepsilon- \boldsymbol{\mathcal N}^\varepsilon.
\end{aligned}
\end{equation*}
The coefficients $\boldsymbol{a}_j$, $\boldsymbol{b}_{jk}$ are given by
\begin{equation*}
	\boldsymbol{a}_j:=\langle \partial_{\xi} \bpsi^\varepsilon_1, \bphi^\varepsilon_j\rangle\,
			\partial_{\xi}\boldsymbol{U}^{\varepsilon}+\partial_\xi \bphi^\varepsilon_j,
	\qquad
	\boldsymbol{b}_{j\ell}:=\langle \partial_{\xi} \bpsi^\varepsilon_1, \bphi^\varepsilon_\ell\rangle
			\,\partial_\xi \bphi^\varepsilon_j.
\end{equation*}
Convergence of the series is guaranteed by assumption \eqref{derpsiphi}. Now let us set
\begin{equation*}
	E_k(s,t):=\exp\left( \int_s^t \lambda_k^\varepsilon(\xi(\tau))d\tau\right).
\end{equation*}
Note that, for $0\leq s<t$, there hold
\begin{equation*}
	E_k(s,t)=\frac{E_k(0,t)}{E_k(0,s)}
		\qquad\textrm{and}\qquad	
	0\leq E_k(s,t)\leq e^{\Lambda_k(t-s)},
\end{equation*}
where $\Lambda_k^\varepsilon := \sup\limits_{\xi\in J} \lambda_k^\varepsilon (\xi) $. From equalities \eqref{eqwk_bis} and and since there holds $v_1 \equiv 0$,  there follows
\begin{equation*}
	\begin{aligned}
	v_k(t)&=v_k(0)\,E_k(0,t)\\
		& \quad+\int_0^t \Bigl\{\langle \bpsi^\varepsilon_k,\boldsymbol{H}^\varepsilon\rangle
			-\theta^\varepsilon(\xi)\sum_{j}\Bigl(\langle \bpsi^\varepsilon_k, \boldsymbol{a}_j\rangle 
			+\sum_{\ell} \langle \bpsi^\varepsilon_k, \boldsymbol{b}_{j\ell}\rangle \,v_\ell\Bigr) v_j
			\Bigr\} E_k(s,t)\,ds \\
		& \quad+\int_0^t  \Bigl\{\langle \bpsi^\varepsilon_k,\boldsymbol{\mathcal Q}^\varepsilon\rangle-\langle \bpsi^\varepsilon_k,\boldsymbol{\mathcal N}^\varepsilon\rangle \bigr\} E_k(s,t) \, ds
	\end{aligned}
\end{equation*}
for $k\geq 2$.
Now  let us introduce the function 
\begin{equation*}
	\boldsymbol{z}(x,t):=\sum_{k\geq 2}v_k(0)\,E_k(0,t)\,\bphi^\varepsilon_k(x;\xi(t)),
\end{equation*}
which satisfies the estimate $|\boldsymbol{z}|_{{}_{L^2}}\leq |\boldsymbol{v}_0|_{{}_{L^2}} e^{\Lambda^\varepsilon_2\, t}$. From the representation formulas for the coefficients $v_k$, there holds
\begin{equation*}
|\boldsymbol{v}-\boldsymbol{z}|^2_{{}_{L^2}} \leq \sum_{k \geq 2} \left\{  \int_0^t \Bigl ( |\langle \bpsi_k^\varepsilon, \boldsymbol{F} \rangle|+ |\langle \bpsi_k^\varepsilon, \boldsymbol{G}\rangle|\Bigr)E_k(s,t) \, ds\right\}^2.
\end{equation*}
Moreover, since
\begin{equation*}
 	|\theta^\varepsilon(\xi)|\leq C\,\Omega^\varepsilon(\xi)
		\qquad\textrm{and}\qquad
	|\langle \bpsi^\varepsilon_k,\boldsymbol{H}^\varepsilon\rangle| \leq C\,\Omega^\varepsilon(\xi)
		\left\{1+ |\langle \bpsi_k^\varepsilon, \partial_\xi \boldsymbol{U}^\varepsilon\rangle|\right\}
\end{equation*}
for some constant $C>0$ depending on the $L^\infty-$norm of $\bpsi^\varepsilon_k$,
there holds
\begin{equation*}
	\begin{aligned}
	\sum_{k \geq 2} \Bigl(  \int_0^t &  |\langle \bpsi_k^\varepsilon, \boldsymbol{F} \rangle|  \, E_k(s,t) \, ds \Bigr)^2 \leq \\
	& \leq C\sum_{k\geq 2}\Bigl(\int_0^t
		 \Omega^\varepsilon(\xi)\Bigl(1+ |\langle \bpsi_k^\varepsilon, \partial_\xi \boldsymbol{U}^\varepsilon\rangle|
		+|\langle \bpsi^\varepsilon_k,\partial_{\xi}\boldsymbol{U}^{\varepsilon}\rangle|
			\sum_{j}|\langle \partial_{\xi} \bpsi^\varepsilon_1, \bphi^\varepsilon_j\rangle||v_j|\\
	&\quad
		+\sum_{j}|\langle \partial_\xi \bpsi^\varepsilon_k,  \bphi^\varepsilon_j\rangle||v_j|
		+\sum_{j} |\langle \bpsi^\varepsilon_k,\partial_\xi \bphi^\varepsilon_j\rangle|\,|v_j|
		\sum_{\ell}|\langle \partial_{\xi} \bpsi^\varepsilon_1, \bphi^\varepsilon_\ell\rangle|\,|v_\ell|\Bigr)
			E_k(s,t)\Bigr)^2\\
	&
	\leq C\sum_{k\geq 2}\Bigl(\int_0^t
		 \Omega^\varepsilon(\xi)\bigl(1+|\boldsymbol{v}|_{{}_{L^2}}^2\bigr)E_k(s,t)\,ds\Bigr)^2.
	\end{aligned}
\end{equation*}
On the other side, concerning the nonlinear terms, there holds
\begin{equation*}
\begin{aligned}
\sum_{k \geq 2} \Bigl(  \int_0^t  & |\langle \bpsi_k^\varepsilon, \boldsymbol{G} \rangle| \, E_k(s,t) \, ds \Bigr)^2 \leq  \\ & \leq C \sum_{k \geq 2} \Bigl( \int_0^t|\langle\bpsi_k^\varepsilon, \boldsymbol{\mathcal Q}^\varepsilon \rangle|  +\left|  \frac{\langle \bpsi_1^\varepsilon,\boldsymbol{\mathcal Q}^\varepsilon \rangle}{1-\langle \partial_\xi \bpsi_1^\varepsilon,\boldsymbol{v}\rangle}\right|\,\Bigl( \sum_{j} |\langle\partial_\xi \bphi^\varepsilon_j, \bpsi_k^\varepsilon \rangle| |\boldsymbol{v}_j|+ |\langle \bpsi_k^\varepsilon, \partial_\xi \boldsymbol{U}^\varepsilon\rangle| \Bigl) \\
&\quad+ \Omega^\varepsilon(\xi) \, \left|  \frac{\langle\partial_\xi \bpsi^\varepsilon_1,\boldsymbol{v} \rangle^2}{1-\langle \partial_\xi \bpsi_1^\varepsilon,\boldsymbol{v}\rangle}\right| \, \Bigl( \sum_{j} |\langle\partial_\xi \bphi^\varepsilon_j, \bpsi_k^\varepsilon \rangle| |\boldsymbol{v}_j|+ |\langle \bpsi_k^\varepsilon, \partial_\xi \boldsymbol{U}^\varepsilon\rangle|\Bigl) \, ds \, \Bigr)^2.
\end{aligned}
\end{equation*}
Moreover, since $|\boldsymbol{\mathcal Q}^\varepsilon|_{{}_{L^1}} \leq C |\boldsymbol{v}|^2_{{}_{L^2}}$, we have
\begin{equation*}
      \begin{aligned}
      |\langle\bpsi_k^\varepsilon, \boldsymbol{\mathcal Q}^\varepsilon \rangle| &\leq C |\boldsymbol{v}|^2_{{}_{L^2}}, \\
       \Bigl|\sum_{k} \langle \bpsi_k^\varepsilon, \partial_\xi \bphi_j^\varepsilon\rangle \boldsymbol{v}_j\Bigr|& \leq C |\boldsymbol{v}|_{{}_{L^2}}, \\
      |\langle \bpsi_k^\varepsilon, \partial_\xi \boldsymbol{U}^\varepsilon\rangle|\, \frac{|\langle \bpsi_1^\varepsilon,\boldsymbol{\mathcal Q}^\varepsilon \rangle|}{|1-\langle \partial_\xi \bpsi_1^\varepsilon,\boldsymbol{v}\rangle|} & \leq \frac{C|\boldsymbol{v}|^2_{{}_{L^2}}}{1-C|\boldsymbol{v}|_{{}_{L^2}} }\leq 2 C |\boldsymbol{v}|^2_{{}_{L^2}}, \\
      |\theta^\varepsilon\langle \partial_\xi \bpsi_1^\varepsilon, \boldsymbol{v} \rangle^2 \, \langle \bpsi_k^\varepsilon, \partial_\xi\boldsymbol{ U}^\varepsilon\rangle| &\leq C\, \Omega^\varepsilon(\xi) |\boldsymbol{v}|^2_{{}_{L^2}},
      \end{aligned}
\end{equation*}
so that we end up with
\begin{equation*}
      \begin{aligned}
      \sum_{k \geq 2} \Bigl(  \int_0^t  |\langle \bpsi_k^\varepsilon, \boldsymbol{G} \rangle| \, E_k(s,t) \, ds \Bigr)^2 &\leq
 C \sum_{k \geq 2} \Bigl( \int_0^t |\boldsymbol{v}|_{{}_{L^2}}^2 (1+ \Omega^\varepsilon(\xi)) \, E_k(s,t) \, ds\Bigr)^2.   
    \end{aligned}
\end{equation*}
Since $\sqrt{a+b}\leq \sqrt{a}+\sqrt{b}$, we infer
\begin{equation*}
	\begin{aligned}
	|\boldsymbol{v}-\boldsymbol{z}|_{{}_{L^2}}&\leq 
		\sum_{k\geq 2}\int_0^t\left \{\Omega^\varepsilon(\xi)
			\bigl(1+|\boldsymbol{v}|_{{}_{L^2}}^2\bigr)+ |\boldsymbol{v}|_{{}_{L^2}}^2 \right \}E_k(s,t)\,ds\\
				&\leq 
		C\int_0^t \Bigl\{\Omega^\varepsilon(\xi)\bigl(1+|\boldsymbol{v}|_{{}_{L^2}}^2\bigr) + |\boldsymbol{v}|^2_{{}_{L^2}}\Bigr\}\,\sum_{k\geq 2} E_k(s,t)\,ds.
	\end{aligned}
\end{equation*}
The assumption on the asymptotic behavior of the eigenvalues $\lambda^\varepsilon_k$
can now be used to bound the series.
Indeed, from the definition of $E_k$, there holds
\begin{equation*}
	\sum_{k\geq 2} E_k(s,t) \leq E_2(s,t) \sum_{k\geq 2} \frac{E_k(s,t)}{E_2(s,t)}
		\leq C\,(t-s)^{-1/2}\,E_2(s,t), 
\end{equation*}
and we infer
\begin{align*}
	 |\boldsymbol{v}-\boldsymbol{z}|_{{}_{L^2}}
		&\leq  C\int_0^t \Omega^\varepsilon(\xi)(t-s)^{-1/2}\,E_2(s,t) \, ds\\
		&\quad +\int_0^t \Bigl\{ |\boldsymbol{v}-\boldsymbol{z}|^2_{{}_{L^2}}+|\boldsymbol{z}|^2_{{}_{L^2}} \Bigr\}\,(t-s)^{-1/2}\,E_2(s,t)\,ds.
	\end{align*}
Now, setting $N(t):=  \sup\limits_{s\in[0,t]} E_1(s,0)\, |(\boldsymbol{v}-\boldsymbol{z})(s)|_{{}_{L^2}}$, we obtain 
\begin{equation}\label{stimalunga}
	\begin{aligned}
	E_1(t,0)\,|\boldsymbol{v}-\boldsymbol{z}|_{{}_{L^2}} &\leq C \int_0^t N^2(s) (t-s)^{-1/2}\,E_2(s,t) \,E_1(0,s)\,ds\\
	& \quad + C\int_0^t \Omega^\varepsilon(\xi)(t-s)^{-1/2}\, E_2(s,t)\, E_1(s,0)\,ds \\
		& \quad + C\int_0^t  |\boldsymbol{v}_0|^2_{{}_{L^2}}
		(t-s)^{-1/2}\,E_2(s,t)\, E_1(s,0)\,ds\\
		&\leq C_1 N^2(t) \, E_1(0,t)+ C_2 \left( |\Omega^\varepsilon|_{{}_{L^\infty}} + |\boldsymbol{v}_0|^2_{{}_{L^2}} E_1(t,0)\right),
	\end{aligned}
\end{equation}
where we used
\begin{equation}\label{disug}
\begin{aligned}
	&\int_0^t e^{\Lambda^\varepsilon_2 s}\,ds
		=\frac{1}{\Lambda_2^\varepsilon}(e^{\Lambda^\varepsilon_2 s}-1)\leq  \frac{1}{|\Lambda_2^\varepsilon|},\\
	&\int_0^t (t-s)^{-1/2}\,E_2(s,t)\,ds
		\leq \int_0^t (t-s)^{-1/2}\,e^{\Lambda_2^\varepsilon\,(t-s)}\,ds
		\leq \frac{1}{|\Lambda_2^\varepsilon|^{1/2}},
\end{aligned}
\end{equation}
and $C_1$ and $C_2$ depend on $\Lambda_2^\varepsilon$. 

Taking the supremum in  \eqref{stimalunga}, we end up with the estimate 
\begin{equation*}
	N(t) \leq A N^2(t)+B
		\qquad\textrm{with}\quad
		\left\{\begin{aligned}
		A&:=C_1  \, E_1(0,t),\\
		B&:=C_2 \left( |\Omega^\varepsilon|_{{}_{L^\infty}} + |\boldsymbol{v}_0|^2_{{}_{L^2}} E_1(t,0)\right).
		\end{aligned}\right.
\end{equation*}
Since $N(0)=0$, if $1-4AB>0$, then
\begin{equation*}
	N<\frac{1-\sqrt{1-4AB}}{2A}=\frac{2B}{1+\sqrt{1-4AB}}\leq 2B.
\end{equation*}
Hence, as soon as 
\begin{equation}\label{finalt}
4 C_1 C_2  \left( E_1(0,t)\,|\Omega^\varepsilon|_{{}_{L^\infty}} + |\boldsymbol{v}_0|^2_{{}_{L^2}} \right) <1
\end{equation}  
we obtain the following $L^2-$estimate for the difference $\boldsymbol{v}-\boldsymbol{z}$
\begin{equation}\label{estfin}
|\boldsymbol{v}-\boldsymbol{z}|_{{}_{L^2}} \leq  \left( E_1(0,t) \, |\Omega^\varepsilon|_{{}_{L^\infty}}  + |\boldsymbol{v}_0|^2_{{}_{L^2}} \right).
\end{equation}
Condition \eqref{finalt} is a condition on the final time $T^\varepsilon$. Indeed,  \eqref{finalt} can be rewritten as
\begin{equation*}
e^{\Lambda_1^\varepsilon t} \leq C \ln \left( \frac{1-|\boldsymbol{v}_0|^2_{{}_{L^2}}}{|\Omega^\varepsilon|_{{}_{L^\infty}}}\right),
\end{equation*}
that is, $T^\varepsilon$ can be chosen of order  $\ln |\Omega^\varepsilon|^{-1}_{{}_{L^\infty}} |\Lambda^\varepsilon_1|^{-1}$.

\end{proof}

\subsubsection{\bf Comments on the proof of Theorem \ref{teo1}}

\begin{remark}{\rm
In the case of the Allen-Cahn equation, $|\Lambda^\varepsilon_1|^{-1} \sim e^{1/\varepsilon}$, so that the order of the final time $T^\varepsilon$ coincides with the corresponding expression determined in \cite{CarrPego89}.

}
\end{remark}

\begin{remark}{\rm

The key idea in the above proof is to write the solution $\boldsymbol{v}$ as a series using the eigenfunctions of the linearized operator $\mathcal L^\e_\xi$ and its adjoint. 

Of course, if $\mathcal L^\e_\xi$ is self-adjoint (as, for instance, in the case of the Allen-Cahn equation, see Example \ref{ExAC}), then the eigenvalues are semi-simple by definition and  assumption {\bf H2} assures that  we can  write the solution $\boldsymbol{v}$ as in \eqref{normalform}; indeed, let us set $\mathcal H= [L^2(I)]^n$ and consider a linear operator $L$ with discrete spectrum and such that $L=L^*$. From the theory of linear self-adjoint operators \cite{Kato}, we have the spectral decomposition
\begin{equation*}
L=\sum_{k \in \N} \lambda_k P_k,
\end{equation*}
where $\{\lambda_k\}_{k \in \N} \subseteq \R$ are the eigenvalues of $L$ and the eigenprojections $P_k$ satisfy
\begin{equation*}
P_k^2=P_k, \quad P_k^*=P_k, \quad \mathds{I}={\sum_{k \in \N}}P_k\quad  \mbox{and}\quad  P_iP_j= \delta_{ij}P_i.
\end{equation*}
The last two properties mean that we have a resolution of the identity and that $P_k$ are mutually orthogonal.
In particular, if we define $ \{V_k\}_{k \in \N}$, the subspaces associated to each $P_k$ (that is, $P_k$ denotes the orthogonal projection onto $V_k$), we can define  the eigenfunctions $\{ \bpsi_{k,j} \}_{j=1}^{g_k} \subseteq V_k$, being $g_k$ the geometric multiplicity of the corresponding eigenvalue $\lambda_k$ (which is equal to its algebraic multiplicity, since $L$ is self-adjoint). 

Given $\boldsymbol{v} \in \mathcal H$ we thus have
\begin{equation*}
T \boldsymbol{v}= \sum _{k \in N} \lambda_k P_k \boldsymbol{v}, \quad \mbox{with} \quad P_k \boldsymbol{v} = \sum_{j=1}^{g_k} \langle \bpsi_{k,j},\boldsymbol{v} \rangle \bpsi_{k,j}.
\end{equation*} 
In particular, this implies that each element of the space $\mathcal H$ can be written as
\begin{equation*}
\boldsymbol{v}=\sum_{k \in \N} P_k \boldsymbol{v},
\end{equation*}
which is exactly \eqref{normalform} with $\bphi_k^\e\equiv \bpsi_k^\e$. Again, we underline that the assumption of having a discrete spectrum is here crucial in order to have a sum (and not an integral coming from the continuous spectrum).
For more details, see \cite[Chapter V, \textsection 3]{Kato}).

\vskip0.2cm

As already pointed out, the linearized operator $\mathcal L^\e_\xi$  is not necessarily self-adjoint (see Example \ref{ExAC} in the case $n\geq2$); however,  it can been  proven that it has a {\it compact resolvent}. To this end, let us set 
\begin{equation*}
\mathcal L^\varepsilon_\xi \boldsymbol{v} :=\left[ \e \d_x^2 \mathds{I}-\mathds V(x)\right] \boldsymbol{v},
\end{equation*}
where $\mathds V(x)$ is a multiplication operator. For example,  when $n=2$, we have
\begin{equation*}
\mathds V(x) \left(\begin{aligned} &w \\ &z \end{aligned}\right)=\left( \begin{aligned}  &\partial_u f_1(U^\varepsilon, V^\varepsilon)  \ \ \ \partial_v f_1 (U^\varepsilon, V^\varepsilon)  \\
 &\partial_u f_2(U^\varepsilon, V^\varepsilon)  \ \ \ \partial_v f_2 (U^\varepsilon, V^\varepsilon) \end{aligned}\right) \left(\begin{aligned} &w \\ &z \end{aligned}\right).
\end{equation*}
If we set $\mathcal L_0:=  \d_x^2 \mathds{I}$, then, given $z \in \rho(\mathcal L_0)$,  the operator $ (\mathcal L_0-z \mathds{I})^{-1}$ is compact, implying that, for $w \in \rho(\mathcal L_0) \cap \rho(\mathcal L_\xi^\e)$, $(\mathcal L_\xi^\e - w \mathds{I})^{-1}$ is compact. 

For the precise proof of this statement, we refer in particular to \cite[Theorem VII-2.4]{Kato}. Precisely, the key point is that, for sufficiently smooth coefficients, the operator $\mathds{V}$ is infinitesimally small (or {\it Kato small}) with respect to $\mathcal L_0$ (for the precise definition, we refer to  \cite[Chapter IV, \textsection 1.1]{Kato}, formula (1.1) with $b_0=0$), implying that  $T(\mu)= \mathcal L_0 + \mu \mathds{V}$ is an holomorphic family of type $(A)$ (see \cite[Chapter VII,  \textsection 2]{Kato}), and we can use Theorem VII-2.4 to prove that $T(\mu)$ has  compact resolvent for all $\mu \in \C$, since this is true for $\mu=0$.
We also refer the reader to \cite{ReedSimon2, ReedSimon4}.

\vskip0.2cm
To conclude, let us  show that if $\mathcal L^\e_\xi$ has a compact resolvent, then it has a spectral resolution. We consider the operator
 \begin{equation}\label{compact}
 T:=(\mathcal L^\e_\xi-z \mathds{I})^{-1}, \quad z \in \rho(\mathcal L^\e_\xi)
 \end{equation}
which belongs to $B_{\infty}(\mathcal H)$ by assumption. As a consequence, there exist $\{ \bphi_k\}_{k \in \N}$ and $\{ \bpsi_k\}_{k \in \N}$ such that
\begin{equation*}
T = \sum_{k \in \N} \mu_k \langle \bpsi_k,  \cdot \rangle \bphi_k,
\end{equation*}
being $\{\mu_k\}_{k \in \N}$ the eigenvalues of $T$. Because of \eqref{compact}, it holds
\begin{equation*}
(\mathcal L^\e_\xi-z \mathds{I})^{-1} = \sum_{k \in \N} \frac{1}{\lambda_k-z} \langle \bpsi_k,  \cdot \rangle \bphi_k,
\end{equation*}
where $\{\lambda_k\}_{k \in \N}$ are the eigenvalues of $\mathcal L^\e_\xi$. 
By applying to the previous identity the operator $(\mathcal L^\e_\xi-z \mathds{I})$, we get
\begin{equation*}
\mathds{I}=\sum_{k \in \N}  \langle \bpsi_k,  \cdot \rangle \bphi_k,
\end{equation*}
leading to
\begin{equation*}
\mathcal L^\e_\xi=\sum_{k \in \N} \lambda_k \langle \bpsi_k,  \cdot \rangle \bphi_k.
\end{equation*}
For further details, we refer the reader to \cite[Chapter III, \textsection 6.8]{Kato} and \cite[Chapter V, \textsection 2.3]{Kato}.
\vskip0.2cm
We now use the assumption that the spectrum of $\mathcal L^\e_\xi$ is  composed only by semi-simple eigenvalues;  it implies (see, for instance, \cite[Definition A.5]{HI}) that for each eigenvalue $\lambda_k$,  the geometric and the algebraic multiplicities are the same and we can thus write 
\begin{equation*}
P_k \boldsymbol{v} = \sum_{j=1}^{g_k} \langle \bpsi_{k,j}, \boldsymbol{v} \rangle \bphi_{k,j} \quad \Rightarrow \quad  \boldsymbol{v}= \sum_{k \in \N}\sum_{j=1}^{g_k} \langle \bpsi_{k,j}, \boldsymbol{v} \rangle \bphi_{k,j} 
\end{equation*}
which is exactly \eqref{normalform} with $v_k:= \sum_{j=1}^{g_k}\langle \bpsi_{k,j}, \boldsymbol{v} \rangle$.

}
\end{remark}

\begin{remark}{\rm
Let us underline that,  since the operator $\mathcal L^\e_\xi$ has a compact resolvent,  then its spectrum is discrete (with no accumulation point different from  $ \infty$), and assumption {\bf H2} actually  reduces to the request of having semi-simple eigenvalues.

Moreover, when $\mathcal L^\e_\xi$ is self-adjoint, then the eigenvalues are semi-simple by definition, and  {\bf H2} is then trivially satisfied.

}
\end{remark}


\begin{remark}{\rm
 As already pointed out, there are no requests on the sign of the first eigenvalue $\lambda^\varepsilon_1$, and up to now we only took advantage of his smallness with respect to $\e$. In particular, if $\lambda^\varepsilon_1$ is positive, as in the case of the Allen-Cahn equation with Neumann boundary conditions, we can't take any advantage in term of convergence from the quantity $E(0,t)$, so that, in order to use estimate \eqref{estfin} to prove  that the perturbation is small in the limit $\e \to 0$,  we need to require the additional smallness assumption on the initial datum $|\boldsymbol{v}_0|_{{}_{L^2}} \leq c \, \varepsilon$. This is however compatible  with the results of \cite{CarrPego89}.

 On the contrary, if the first eigenvalue  is small but  negative (which translates into the fact that the solution around which we are linearizing is an approximation of a  stable configuration for the system), by slightly modifying the end of the proof of Theorem \ref{teo1}, we end up with the following estimate for the perturbation
 \begin{equation*}
|\boldsymbol{v}-\boldsymbol{z}|_{{}_{L^2}} \leq  \left(  |\Omega^\varepsilon|_{{}_{L^\infty}}  + E_1(0,t) \,|\boldsymbol{v}_0|^2_{{}_{L^2}} \right).
\end{equation*}
In this case we can ensure the convergence of the term $E_1(0,t) \,|\boldsymbol{v}_0|^2_{{}_{L^2}}$ for $t \to \infty$, so that we have no longer to require any smallness assumption on the initial datum $\boldsymbol{v}_0$. 
Clearly, also the condition on the final time $T^\varepsilon$ changes, but a straightforward computation shows that it turns to be again of order $|\Lambda^\varepsilon_1|^{-1}$.

This is what happens when other types of stable equilibria than the constant states exists, as, for example, in the case of the Allen-Cahn equation with Dirichlet boundary conditions (see \cite{BerBraButPre}); in this case we expect the time dependent solution to converge, exponentially slow, to such non-costant configurations, rather then to a patternless steady state.


}
\end{remark}

\begin{remark}
{\rm Another interesting remark on the estimate \eqref{estfin} is the dependence on $\Lambda_2^\e$ of the quantities on the right hand side; such dependence is encoded in the constant $C_2$ and comes out by the use of the inequalities \eqref{disug}.

Indeed, exploiting all the computations, we actually end up with the following bound for $\boldsymbol{v}$:
 \begin{equation*}
|\boldsymbol{v}-\boldsymbol{z}|_{{}_{L^2}} \leq  \left(   |\Lambda_2^\e|^{-1/2} \, |\Omega^\varepsilon|_{{}_{L^\infty}}  +   |\Lambda_2^\e|^{-3/2} \, E_1(0,t) \,|\boldsymbol{v}_0|^2_{{}_{L^2}} \right),
\end{equation*}
so that it is interesting, in term of convergence of the quantity $|\boldsymbol{v}-\boldsymbol{z}|_{{}_{L^2}}$ in the limit $\e \to 0$, to understand how the term $ |\Lambda_2^\e|$ behaves with respect to $\e$. As already pointed out, in the case of the initial boundary value problem \eqref{cauchyAC}, $\lambda^\e_k \leq -c \, $ for all $k \geq 2$, so $\Lambda_2^\e$ is in fact a constant independent on $\e$ and it does not help in terms of convergence. On the other side, the are different examples (see, for instance, \cite{Str13}), where such a quantity behaves like $\e^{-\alpha}$, $\alpha>0$: this is meaningful  since we would have the second term on the right hand side to converge to zero as $\e \to 0$ even in the case $\lambda_1^\e>0$ and without any smallness assumption on the initial datum $\boldsymbol{v}_0$.

}
\end{remark}

\subsection{The slow convergence of the solution}

Estimate \eqref{estfin} can be used to decouple the system \eqref{NLS}; this leads to the following consequence of Theorem \ref{teo1}.

\begin{proposition}\label{cor:metaL2}
Let hypotheses {\bf H1-2-3} be satisfied and let us also assume
\begin{equation}\label{asstheta}
	(\xi-\bar \xi)\,\theta^\varepsilon(\xi)<0\quad\textrm{ for any } \xi\in J,\,\xi\neq \bar \xi
	\qquad\textrm{ and }\qquad
	{\theta^\varepsilon}'(\bar \xi)<0.
\end{equation}
Then, for $\varepsilon$ and $|\boldsymbol{v_0}|_{{}_{L^2}}$ sufficiently small, the solution $(\xi,v)$ converges exponentially fast to $(\bar \xi,0)$ as $t\to+\infty$.
\end{proposition}

\begin{remark}{\rm

Since the  sign of $\theta^\e$ somehow gives the sign of the speed of the interface convergence, assumption \eqref{asstheta} describe the observable fact that the interface moves  either towards $x=-\ell$ or $x=\ell$ depending  on whether the solution is converging towards $\pm \boldsymbol {u^*}$ (corresponding to the equilibrium locations $\bar \xi = \pm \ell$). For example, if the layer is closer to $x=\ell$ then it will converge towards $\ell$ with positive speed; on the contrary, if for small times the interface location il closer to $-\ell$, then we will have convergence  towards $-\ell$ for the variable $\xi$ (corresponding, as already pointed out, to the convergence of the solution towards $- \boldsymbol{u^*}$).
}
\end{remark}

\begin{proof}[\bf Proof of Proposition \ref{cor:metaL2}]
For any initial datum $\xi_0$, the variable $\xi(t)$ is such that
\begin{equation*}
	\frac{d\xi}{dt}=\theta^\varepsilon(\xi)\bigl(1 + r\bigr)+ \rho^\varepsilon(\xi,\boldsymbol{v}),
		\qquad\textrm{with}\quad
	|r|\leq C\bigl\{|\boldsymbol{v}_0|^2_{{}_{L^2}}\left( e^{-ct}+1\right)+|\Omega^\varepsilon|_{{}_{\infty}}\bigr\}
\end{equation*}
and 
$$ |\rho^\varepsilon(\xi,\boldsymbol{v})| \leq C |\boldsymbol{v}|^2_{{}_{L^2}} \leq |\boldsymbol{z}|^2_{{}_{L^2}}+ R \leq |\boldsymbol{v}_0|^2_{{}_{L^2}} e^{-ct} + |\boldsymbol{v}_0|^2_{{}_{L^2}}+|\Omega^\varepsilon|_{{}_{\infty}}. $$
Hence, recalling that $|\boldsymbol{v}_0|^2_{{}_{L^2}} \leq c \, \varepsilon$, in the regime of small $\varepsilon$ the solution $\xi(t)$ has similar decay properties of the solution to the following reduced equation
\begin{equation*}
\frac{d\eta}{dt}=\theta^\varepsilon(\eta), \qquad  \eta(0)=\xi(0).
\end{equation*}
As a consequence, $\xi$ converges exponentially fast to $\bar \xi$ as $t\to+\infty$. More precisely there exists $\beta^\varepsilon>0$, $\beta^\varepsilon \to 0$ as $\varepsilon \to 0$, such that 
\begin{equation}\label{finalestxi}
|\xi(t)- \bar \xi| \leq |\xi_0|e^{-\beta^\varepsilon t},
\end{equation}
for any $t$ under consideration. Concerning the perturbation $\boldsymbol{v}$, from \eqref{eqwk_bis}, we deduce
\begin{equation*}
\begin{aligned}
	\boldsymbol{v}_k(t)=\boldsymbol{v}_k(0)\exp\left(\int_{0}^t \lambda^\varepsilon_k d\tau\right)
		&+\int_{0}^{t} \langle \bpsi^\varepsilon_k,\boldsymbol{F}\rangle(s)
		\exp\left(\int_{s}^t \lambda^\varepsilon_k d\tau\right)ds \\
		&+\int_{0}^{t} \langle \bpsi^\varepsilon_k,\boldsymbol{G}\rangle(s)
		\,\exp\left(\int_{s}^t \lambda^\varepsilon_k\,d\tau\right)ds.
		\end{aligned}
\end{equation*}
By summing on $k$ and from the Jensen's inequality, we get
\begin{equation*}
	\begin{aligned}
	|\boldsymbol{v}|_{{}_{L^2}}^2(t)\leq C\left\{|\boldsymbol{v}_0|_{{}_{L^2}}^2\,e^{2\Lambda^\varepsilon_1\,t}
		+\int_{0}^{t} \left( |\boldsymbol{F}|_{{}_{L^2}}^2(s)+ |\boldsymbol{G}|_{{}_{L^2}}^2(s) \right) \,e^{2\Lambda^\varepsilon_1(t-s)}\,ds\right\}.
	\end{aligned}
\end{equation*}
Let $\nu^\varepsilon>0$ and $\mu^\varepsilon>0$ be such that $|F|_{{}_{L^2}}(t)\leq C\,e^{-\nu^\varepsilon\,t}$ and $|G|_{{}_{L^2}}(t)\leq C\,e^{-\mu^\varepsilon\,t}$; 
then, if $\nu^\varepsilon, \mu^\varepsilon \neq |\Lambda_1^\varepsilon|$, there holds
\begin{equation}\label{finalev}
	 |\boldsymbol{v}|_{{}_{L^2}}^2(t)\leq C\left\{|\boldsymbol{v}_0|_{{}_{L^2}}^2\,e^{2\Lambda^\varepsilon_1\,t}
		+t\left(e^{-2\nu^\varepsilon \,t}+ e^{-2\mu^\varepsilon \,t}\right)\right\},
\end{equation}
showing the exponential convergence to $0$ of the perturbation $\boldsymbol{v}$.
\end{proof}
\vskip0.2cm
Theorem \ref{teo1} and Proposition \ref{cor:metaL2} give a complete and rigorous description of the dynamics of the solutions to \eqref{RDsystem}, up to its very first stage, when the internal interface is formed over the initial datum; indeed, estimate \eqref{estfin} together with the precise order of the time $T^\e$ states that the layered solution $u$ remains close to an unstable configuration for a time  $T^\e$ that can be extremely long as the parameter $\e \to 0$. After this long phase of the dynamics, the perturbation $v$ becomes neglectable (as pointed out in \eqref{finalev}), and the solution $u$ is approximately given by $U^\e(x;\xi(t))$, where the parameter $\xi$ evolves accordingly to \eqref{finalestxi}: such formula shows the slow motion of $u$ throughout the slow motion of the interface. Indeed, \begin{equation*}
\xi(t) \sim \bar \xi +  |\xi_0|e^{-\beta^\varepsilon t},
\end{equation*}
meaning that the layer is converging exponentially slow towards its asymptotic configuration, and this motion is much slower as $\varepsilon$ becomes smaller. In this last phase of the dynamics the layer suddenly  disappears  (we can think as it is collapsing toward one of the walls $x=\pm \ell$) and we have convergence of the solution to one of the patternless equilibrium configuration $\pm \boldsymbol{u^*}$.


\vskip3cm

\begin{thebibliography}{99}

\bibitem{AlaBroGui96}
Alama S., Bronsard L., Gui C.,
{Stationary layered solutions in $\R^2$ for an Allen-Cahn system with multiple well potential},
Calc. Var. Partial Differential Equations 5 (1997), no. 4, 359--390.


\bibitem{AlikBatFus91}
Alikatos N.D., Bates P.W., Fusco G.,
{\it Slow motion for the Cahn-Hilliard equation in one space dimension},
J. DIfferential Equations 90 (1991) no. 1, 81--135

\bibitem{AliFus94}
Alikakos N. D., Fusco G., 
{\it The Spectrum of the Cahn-Hilliard operator for generic interface in higher space dimensions}, 
Indiana Univ. Math. J. 42, No. 2 (1993), 637--674.

\bibitem{AliFus08}
Alikakos N.D., Fusco G.,
{\it On the connection problem for potentials with several global minima}, 
Indiana Univ. Math. J. 57 (2008), no. 4, 1871Ð1906. 


\bibitem{AliFus96}
Alikakos N.D., Fusco G., Stefanopolous V.,
{\it Critical spectrum and stability of interfaces for a class of reaction-diffusion equations},
J. Diff. Equations 126 (1996), 106--167.



\bibitem{AllCah79}
Allen S. M., Cahn J. W.,
{\it A microscopic theory for anti phase boundary motion and its application to anti phase domain coarsening},
Acta Metall. 27 (1979), 1085--1095.

\bibitem{Ang88}
Angenent S.,  
{\it The zero set of a solution of a parabolic equation}, 
J. Reine Angew. Math. 390 (1988), 79--96. 


\bibitem{AngMalPel87}
Angenent S. B., Mallet-Paret J., Peletier L. A., 
{ \it Stable transition layers in a semilinear boundary value problem}, 
J. Differential Equations 67 (1987), 212--242.

\bibitem{AroWei75}
Aronson D. G., Weinberger H. F., 
{\it Nonlinear diffusion in population genetics, combustion and nerve propagation}, 
Lecture Notes in Mathemarics (1975), no 446, 5--49, Springer, Berlin.

\bibitem{Bar10}
Bartels S.,
{\it A lower bound for the spectrum of the linearized Allen-Cahn operator near a singularity},
Bonn 2010.


\bibitem{BerKamSiv01}
Berestycki, H., Kamin S., Sivashinsky G.,
{\it Metastability in a  flame front evolution equation},
 Interfaces Free Bound. 3 (2001), no. 4, 361--392. 

\bibitem{BerBraButPre} 
 Bertini, L., Brassesco, S., Butt\`a, P.,
 {\it Dobrushin states in the $\phi^4_1$ model.}
  Arch. Ration. Mech. Anal. 190 (2008), no. 3, 477--516.
 
 \bibitem{BetOrlSme11}
 Bethuel F., Orlandi G., Smets D.,
 {\it Slow motion for gradient systems with equal depth multiple-well potentials},
 J. Differential Equations 250 (2011), 53--94.
 
 \bibitem{BetSme13}
  Bethuel F.,  Smets D.,
  {\it Slow motion for equal depth multiple-well gradient systems: the degenerate case},
   Discrete Contin. Dyn. Syst. 33 (2013), no. 1, 67--87.
   
   \bibitem{BroKoh90}
    Bronsard L., Kohn R.V.,
   {\it On the slowness of phase boundary motion in one space dimension},
    Comm. Pure Appl. Math. 43 (1990), no. 8, 983--997.
 
 \bibitem{CagFif}
Caginalp G., Fife P. C., 
{\it Elliptic problems involving phase boundaries satisfying a curvature condition}, 
 IMA J. Appl. Math. 38 (1987), no. 3, 195--217.


\bibitem{CarrPego89}
Carr, J., Pego, R. L., 
{\it Metastable patterns in solutions of $u_t=\varepsilon^2 u_{xx}+f(u)$}, 
Comm. Pure Appl. Math. 42 (1989) no. 5, 523--576. 

\bibitem{CarrPego90}
 Carr, J., Pego, R.,
 {\it Invariant manifolds for metastable patterns in $u_t=\varepsilon u_{xx}-f(u)$}, 
 Proc. Roy. Soc. Edinburgh Sect. A 116 (1990) no. 1-2, 133--160.

\bibitem{Che92}
 Chen X., 
 { \it Generation and propagation of interface in reaction-diffusion equations},
J. Differential Equations 96 (1992), 116--141.
 
\bibitem{Che94} 
Chen X., 
{ \it Spectrums for the Allen-Cahn, Cahn-Hilliard, and phase-field equations for generic interfaces}, 
Comm. Partial Differential Equations 19 (1994), 1371--1395.

\bibitem{Che04}
Chen X.,
{\it Generation, propagation, and annihilation of metastable patterns},
J. Differential Equations 206 (2004), 399--437.

\bibitem{ClePel}
Clement P., Peletier L. A., 
{\it On a nonlinear eigenvalue problem occurring in population genetics}, 
Proc. R. Soc. Edinb. (1985), 85--101.

\bibitem{DemSch90}
DeMottoni P., Schatzman M., 
{\it Development of interfaces in $\R^N$}, 
Proc. Roy. Soc. Edin. Sect. A 116 (1990), 207--220.

\bibitem{Evans92}
Evans L., Soner H., Souganidis P.,
{\it Phase transition and generalized motion by mean curvature},
Comm. Pure Appl. Math 45 (1992), 1097--1123.


\bibitem{Fif74}
Fife P. C., 
{ \it Transition layers in singular perturbation problems},
J. diff. Equations 15 (1974),  77--105.

\bibitem{Fif76}
Fife P. C., 
{ \it Boundary and interior transition layer phenomena for pairs of second order differential equations}, 
J. Math. Analysis Appl. 54 (1976), 497--521.

\bibitem{Fif76b}
Fife P. C.,
{\it  Pattern formation in reacting and diffusing systems},
J. chem. Phys. 64  (1976), 854--864.

\bibitem{Fif02}
Fife P.C.,
{\it Pattern formation in gradient systems},
in "Handbook of Dynamical Systems"  2 (2002), 677--722, ed. B. Fielder, Elsevier Science, Amsterdam, The Netherlands.



\bibitem{FifHsi88}
Fife P. C., Hsiao L., 
{ \it The generation and propagation of internal layers},
J. Nonlinear Analysis 12 (1988), 19--41.


\bibitem{Fus90}
G. Fusco, 
{\it A geometric approach to the dynamics of $u_t=\varepsilon^2u_{xx}+f(u)$ for small $\e$},  
 Problems involving change of type (Stuttgart, 1988), 53--73,
Lecture Notes in Physics 359 (1990)  Springer, Berlin.

\bibitem{FuscHale89}
Fusco, G., Hale, J. K.,
{\it Slow-motion manifolds, dormant instability, and singular perturbations},
J. Dynam. Differential Equations 1 (1989) no. 1, 75--94. 

\bibitem{Gra95}
Grant C.P.,
{\it Slow motion in one-dimensional Cahn-Morral systems},
SIAM J. Math. Anal. 26 (1995), 21--34.



\bibitem{HI}
Haragus M., Iooss G.,
{\it Local Bifurcation, Center Manifolds, and Normal Forms in Infinite-Dimensional Dynamical Systems},
Springer London (2011), Print ISBN  978-0-85729-111-0.


\bibitem{HubeSerr96}
Hubert F., Serre D.,
{\it Fast-slow dynamics for parabolic perturbations of conservation laws},
Comm. Partial Differential Equations 21 (1996) no. 9-10, 1587--1608. 

\bibitem{Kato}
Kato T.,
{\it Perturbation Theory for Linear Operator},
Springer (1995).

\bibitem{Kaletall01}
Kalies W.D., VanderVorst R.C.A.M.,  Wanner T.,
{\it Slow motion in higher-order systems and $\Gamma$-convergence in one space dimension },
Nonlinear Analysis 44 (2001), 33--57.

\bibitem{Kee80}
Keener J. P., 
Waves in excitable media, SIAM J. appl. Math. 39, 528-548 (1980).

\bibitem{KreiKrei86}
Kreiss G., Kreiss H.-O.,
{\it Convergence to steady state of solutions of Burgers' equation},
Appl. Numer. Math. 2 (1986) no. 3-5, 161--179. 

\bibitem{KreiKreiLore08}
Kreiss G., Kreiss H.-O., Lorenz J.,
{\it Stability of viscous shocks on finite intervals}, 
Arch. Ration. Mech. Anal. 187 (2008) no. 1, 157--183. 


\bibitem{LafoOMal94}
Laforgue J.G.L., O'Malley R.E. Jr., 
{\it On the motion of viscous shocks and the supersensitivity of their steady-state limits},
Methods Appl. Anal. 1 (1994), no. 4, 465--487. 

\bibitem{LafoOMal95}
Laforgue J.G.L., O'Malley R.E. Jr., 
{\it Shock layer movement for Burgers equation},
Perturbations  methods in physical mathematics (Troy, NY, 1993). SIAM J. Appl. Math. 55 (1995)  no. 2, 332--347.

\bibitem{MS}
Mascia C., Strani M.,
{\it Metastability for nonlinear parabolic equations with application to scalar conservation laws }, SIAM J. Math. Anal. 45 (2013), no. 5, 3084--3113.

\bibitem{MimMur79}
Mimura M., Murray J. D., 
{\it On a planktonic prey-predator model which exhibits patchiness}, 
J. theor. Biol. 75 (1979), 249--262.

\bibitem{MimTabHos80}
Mimura M.,Tabata M., Hosono Y., 
{\it Multiple solutions of two-point boundary value problems of Nuemann type with a small parameter}, 
SIAM J. Math. Analysis 11  (1980), 613--631.

\bibitem{NisFuj86}
Nishiura Y., Fujii H.,
{\it Stability of singularly perturbed solutions to systems of reaction-diffusion equations},
SIAM J. Math. Anal. 18 (1987), no. 6, 1726--1770.

\bibitem{OrtLos75}
Ortoleva P.,  Loss J., 
{\it Theory of propagation of discontinuities in kinetic systems with multiple time scales: fronts, front multiplicity, and pulses}, 
J. chem. Phvs. 63 (1975), 3398--3408.

\bibitem{OstYan75}
Ostrovskii A.,  Yanho V. G., 
{\it The formation of pulses in an excitable medium}, 
Biofizika 20 (1975), 489--493.

\bibitem{OttRez06}
Otto F., Reznikoff M.G.,
{\it Slow Motion of Gradient Flows},
J, Diff., Equations 237 (2006), 372--420.


\bibitem{Pego89}
Pego R.L.,
{\it Front migration in the nonlinear Cahn-Hilliard equation},
Proc. Roy. Soc. London Ser. A 422 (1989) no. 1863, 261--278.

\bibitem{ReedSimon2}
Reed M., Simon B.,
{\it Methods of Modern Mathematical Physics. II: Fourier Analysis, self-adjointness},
 Academic Press, New York, 1975.
 
 \bibitem{ReedSimon4}
Reed M., Simon B.,
{\it Methods of Modern Mathematical Physics. IV: Analysis of operators},
 Academic Press, New York, 1978.


\bibitem{ReynWard95}
Reyna L.G., Ward M.J.,
{\it On the exponentially slow motion of a viscous shock},
Comm. Pure Appl. Math. 48 (1995), no. 2, 79--120. 

\bibitem{Ris08}
Risler E.,
{\it Global convergence toward traveling fronts in nonlinear parabolic systems with a gradient structure},
Ann. Inst. H. Poincar\'e Anal. Non Lin\'eaire 25 (2008), no. 2, 381--424.

\bibitem{Ste87}
Sternberg, P.,
{\it Vector-valued local minimizers of nonconvex variational problems},
Current directions in nonlinear partial differential equations (Provo, UT, 1987).
Rocky Mountain J. Math. 21 (1991), no. 2, 799--807. 

\bibitem{Str12}
Strani M.,
{\it Slow motion of internal shock layers for the Jin-Xin system in one space dimension}, J. Dyn. Differential Equations 27 (2015),  no. 1, pp. 1--27.

\bibitem{Str13}
Strani M.,
{\it Metastable dynamics of internal interfaces for a convection-reaction-diffusion equation},
Nonlinearity,  28 (2015) 4331--4368.

\bibitem{SunWard99}
Sun X., Ward M. J., Russell R.,
{\it Metastability for a generalized Burgers equation with application to propagating flame fronts},
European J. Appl. Math. 10 (1999), no. 1, 27--53.



\end{thebibliography}
\end{document}